\RequirePackage{amsmath} 
\documentclass[smallcondensed,numbook]{svjour3}

\usepackage{amssymb}
\usepackage{hyperref}
\usepackage{graphicx}
\usepackage{bbm}
\usepackage{mathtools}
\usepackage{mleftright}
\usepackage{physics}
\usepackage{siunitx}
\usepackage{xcolor}
\usepackage{xparse}
\usepackage{pgfplots}
\usepackage{tikz-cd}
\usepackage{csl}
\usepackage[caption=false]{subfig}
\usepackage{cleveref}
\usepackage[disable]{todonotes}

\smartqed
\graphicspath{{./figures/}}

\pgfplotsset{compat=newest}
\pgfplotsset{every axis/.append style={
		x label style={font=\small},
		y label style={font=\small},
		tick label style={font=\small},
		legend style={font=\small}
}}

\newenvironment{butchertableau}[2][1.4]
	{\def\arraystretch{#1}\array{#2}}
	{\endarray}

\NewDocumentCommand{\convplot}{m m m m m s}{
    \begin{tikzpicture}
        \begin{loglogaxis}[
        	    \IfBooleanT{#6}{
        	        legend entries={$M=2$,$M=4$,$M=6$},
        	        legend columns=-1,
        	        legend style={/tikz/every even column/.append style={column sep=0.4cm}},
        	        legend to name=conv,
        	    }
    	        legend style={/tikz/every even column/.append},
        	    xlabel=Steps,
        	    ylabel=Error,
        	    width={#1}
            ]
            \addplot table[x index=0,y index=3] {conv_#2};
            \addplot table[x index=1,y index=4] {conv_#2};
            \addplot table[x index=2,y index=5] {conv_#2};
            \draw (axis cs:#3) -- node[below]{$#5$} (axis cs:#4);
        \end{loglogaxis}
    \end{tikzpicture}
}

\NewDocumentCommand{\wpplot}{m m}{
    \begin{tikzpicture}
        \begin{loglogaxis}[
        	    legend entries={#1},
        	    legend cell align={left},
        	    legend pos=south west,
        	    legend style={draw=none, fill=none, font=\scriptsize},
        	    xlabel={Time (s)},
        	    ylabel={Error},
        	    width={\textwidth},
        	    height={0.4\textwidth}
            ]
            \addplot table[x index=1,y index=2] {perf_#2};
            \addplot table[x index=4,y index=5] {perf_#2};
            \addplot table[x index=7,y index=8] {perf_#2};
        \end{loglogaxis}
    \end{tikzpicture}
}

\DeclareMathOperator{\diag}{diag}
\DeclarePairedDelimiter\floor{\lfloor}{\rfloor}
\newcommand{\R}{\mathbb{R}}
\newcommand{\C}{\mathbb{C}}
\NewDocumentCommand{\eye}{g}{I_{\IfNoValueF{#1}{#1 \times #1}}}
\newcommand{\vecone}{\mathbbm{1}}
\newcommand{\Rint}{R_\text{int}}
\NewDocumentCommand{\cmrgark}{s}{\IfBooleanTF{#1}{C}{c}ompound-fast}

\newcommand{\fast}{\mathfrak{f}}
\newcommand{\slow}{\mathfrak{s}}
\NewDocumentCommand{\component}{m m m}{^{ \left\{ #1 \right\} #2 \IfBooleanT{#3}{T} }}
\NewDocumentCommand{\comp}{m O{} s}{\component{#1}{#2}{#3}}
\NewDocumentCommand{\F}{O{} s}{\component{\fast}{#1}{#2}}
\NewDocumentCommand{\FL}{O{\lambda} O{} s}{\component{\fast, #1}{#2}{#3}}
\RenewDocumentCommand{\S}{O{} s}{\component{\slow}{#1}{#2}}
\NewDocumentCommand{\SL}{O{\lambda} O{} s}{\component{\slow, #1}{#2}{#3}}
\NewDocumentCommand{\FF}{O{} s}{\component{\fast, \fast}{#1}{#2}}
\NewDocumentCommand{\FS}{O{} s}{\component{\fast, \slow}{#1}{#2}}
\NewDocumentCommand{\FSL}{O{\lambda} O{} s}{\component{\fast, \slow, #1}{#2}{#3}}
\NewDocumentCommand{\SF}{O{} s}{\component{\slow, \fast}{#1}{#2}}
\NewDocumentCommand{\SFL}{O{\lambda} O{} s}{\component{\slow, \fast, #1}{#2}{#3}}
\RenewDocumentCommand{\SS}{O{} s}{\component{\slow, \slow}{#1}{#2}}

\begin{document}

\title{Implicit Multirate GARK Methods}

\author{Steven Roberts \and
        John Loffeld \and
        Arash Sarshar \and
        Carol S. Woodward \and
        Adrian Sandu
}

\authorrunning{S. Roberts \and J. Loffeld \and A. Sarshar
               \and C.S. Woodward \and A. Sandu} 

\newcommand{\vtinfo}{\csl{}, \csldepartment{}, \cslinstitution{}}
\newcommand{\llnlinfo}{Lawrence Livermore National Laboratory}

\institute{S. Roberts (corresponding author) \at \vtinfo{} \\ \email{steven94@vt.edu}
           \and
           J. Loffeld \at \llnlinfo{} \\ \email{loffeld1@llnl.gov}
           \and
           A. Sarshar \at \vtinfo{} \\ \email{sarshar@vt.edu}
           \and
           C.S. Woodward \at \llnlinfo{} \\ \email{woodward6@llnl.gov}
           \and
           A. Sandu \at \vtinfo{} \\ \email{sandu@cs.vt.edu} 
}
\date{Received: \today / Accepted: date}

\newif\ifreport
\reporttrue

\ifreport
    \cslauthor{Steven Roberts, John Loffeld, Arash Sarshar, Carol S. Woodward, and Adrian Sandu}
    \cslemail{\{steven94, sarshar\}@vt.edu, \{loffeld1, woodward6\}@llnl.gov, sandu@cs.vt.edu}
	\cslyear{19}
	\cslreportnumber{5}
	\csltitlepage{}
\fi

\maketitle

\begin{abstract}
    This work considers multirate generalized-structure additively partitioned Runge--Kutta (MrGARK) methods for solving stiff systems of ordinary differential equations (ODEs) with multiple time scales.  These methods treat different partitions of the system with different timesteps for a more targeted and efficient solution compared to monolithic single rate approaches.  With implicit methods used across all partitions, methods must find a balance between stability and the cost of solving nonlinear equations for the stages.  In order to characterize this important trade-off, we explore multirate coupling strategies, problems for assessing linear stability, and techniques to efficiently implement Newton iterations for stage equations.  Unlike much of the existing multirate stability analysis which is limited in scope to particular methods, we present general statements on stability and describe fundamental limitations for certain types of multirate schemes.  New implicit multirate methods up to fourth order are derived, and their accuracy and efficiency properties are verified with numerical tests.
\end{abstract}


\section{Introduction}
\label{sec:intro}

In many real-world dynamical systems, there are parts of the system that evolve at significantly faster rates than other parts of the system.  Time integration methods in which a single timestep is applied to all parts of the system can be inefficient and unsatisfactory for these multiscale problems.  The fastest dynamics impose a relatively small, global timestep to ensure stability, to meet accuracy requirements, and in the case of an implicit method, to ensure convergence of the nonlinear solves for stage equations.  This forces the slowest dynamics to be evaluated more frequently than necessary, leading to a costly integration.  Instead of treating such a system as a black box, many numerical methods consider the fast and slow processes independently:
\begin{equation} \label{eq:ode}
    y' = f(y) = f\F(y) + f\S(y), \quad
    y(t_0) = y_0, \quad
    y(t) \in \R^d.
\end{equation}
An important special case of this additively partitioned system is the component partitioned problem
\begin{equation} \label{eq:ode_comp}
    \begin{bmatrix}
        y\F \\ y\S
    \end{bmatrix}' = \begin{bmatrix}
        f\F \mleft( y\F, y\S \mright) \\ f\S \mleft( y\F, y\S \mright)
    \end{bmatrix},
\end{equation}
with $y\F \in \R^{d\F}$, $y\S \in \R^{d\S}$, and $d = d\F + d\S$.

Multirate methods efficiently solve the system of ordinary differential equations (ODEs) given in \cref{eq:ode} by integrating the fast dynamics $f\F$ with a smaller timestep than the slow dynamics $f\S$.  The choice of how to partition $f$ into $f\F$ and $f\S$ can depend on many factors including stiffness, accuracy requirements, evaluation cost, linearity, and memory requirements.  In the case where an implicit method is needed, which will be the focus of this paper, the cost and convergence of the nonlinear solver also comes into consideration.  There may be a small number of components of an ODE that causes slow convergence of Newton's method (e.g. a boundary layer).  Such components can be grouped into $f\F$.  In some cases, the Jacobain of $f$ is an unstructured matrix leading to expensive linear solves, but the problem can be decomposed such that linear solves with the Jacobians of $f\F$ and $f\S$ are inexpensive.  Alternating directions implicit (ADI) methods \cite{peaceman1955numerical} and approximate matrix factorization (AMF) methods \cite{beam1976implicit}, for example, exploit this property.

Implicit methods require excellent stability to offset the cost of solving potentially nonlinear equations in each step.  For this reason, an understanding of the stability of multirate methods is crucial.  One of the first works studying multirate stability was that of Gear~\cite{gear1974multirate}.  Subsequent authors have examined multirate stability in the context of backward Euler~\cite{sand1992stability,skelboe1989stability,verhoeven2006general}, Runge--Kutta methods~\cite{andrus1993stability,kvaerno2000stability,hundsdorfer2009analysis}, linear multistep methods~\cite{gear1984multirate,skelboe1989stability,verhoeven2007stability}, and Rosenbrock methods~\cite{gunther1993multirate,rodriguez2004computing,savcenco2008comparison}.

Much of the development and implementation of multirate schemes for stiff systems has focused on multirate Rosenbrock methods \cite{gunther1993multirate,Gunther1997,Bartel2002,Savcenco2009}, but methods based on implicit Runge--Kutta methods have been explored as well.  In \cite{hundsdorfer2009analysis}, a multirate $\theta$-method is presented and analyzed.  Recently, multirate methods based on TR-BDF2 were proposed in \cite{DelpopoloCarciopolo2018,bonaventura2018self}.  In \cite{sandu2018class,roberts2018coupled}, new strategies for creating implicit multirate infinitesimal methods were introduced.

In~\cite{Sandu_2015_GARK}, Sandu and G\"{u}nther propose the generalized-structure additively partitioned Runge--Kutta (GARK) family of methods.  GARK provides a unifying framework that includes traditional, implicit-explicit (IMEX), and multirate Runge--Kutta methods.  Order conditions as well as the linear and nonlinear stability analysis are developed for this large class of methods.  G\"{u}nther and Sandu continue in~\cite{gunther2016multirate} where many variants of multirate Runge--Kutta methods are cast as GARK methods.  Multirate GARK (MrGARK) methods up to order four are derived in~\cite{sarshar2019design}.  These include methods that are explicit in both partitions and methods that combine explicit and implicit methods.

In this work, we develop new MrGARK methods that are implicit in both the fast and slow partitions.  The development is guided by new theoretical results regarding the stability of multirate methods.  Necessary and sufficient conditions for achieving A-stability are presented, as well as some fundamental stability limitations on certain types of multirate methods.  Many of these results extend past multirate methods to the entire GARK framework.  Numerical experiments verify the order of convergence and the efficiency of the new schemes.

The structure of this paper is as follows.  \Cref{sec:MrGARK} introduces multirate methods using the GARK framework.  The linear stability of multirate methods is explored in \cref{sec:stability}.  \Cref{sec:newton} discusses techniques to efficiently implement the Newton iterations.  \Cref{sec:methods} contains the newly derived implicit MrGARK methods, and \cref{sec:experiments} presents the numerical experiments used to test the methods.  Finally, we summarize the results of the paper in \cref{sec:conclusion}.
\section{Multirate GARK Methods}
\label{sec:MrGARK}

The GARK framework \cite{Sandu_2015_GARK} is used as the foundation for representing and analyzing multirate Runge--Kutta methods.  In the most general form for a two-partitioned system \cref{eq:ode}, one step reads
\begin{subequations} \label{eq:GARK}
    \begin{align}
        \begin{split}
            Y\F_i &= y_n + H \, \sum_{j=1}^{s\F} a\FF_{i,j} \, f\F \mleft( Y\F_j \mright) + H \, \sum_{j=1}^{s\S} a\FS_{i,j} \, f\S \mleft( Y\S_j \mright), \\
            & \quad \text{for } i = 1, \dots, s\F,
        \end{split} \\
        \begin{split}
            Y\S_i &= y_n + H \, \sum_{j=1}^{s\S} a\SS_{i,j} \, f\S \mleft( Y\S_j \mright) + H \, \sum_{j=1}^{s\F} a\SF_{i,j} \, f\F \mleft( Y\F_j \mright), \\
            & \quad \text{for } i = 1, \dots, s\S,
        \end{split} \\
        y_{n+1} &= y_n + H \, \sum_{j=1}^{s\F} b\F_j \, f\F \mleft( Y\F_j \mright) + H \, \sum_{j=1}^{s\S} b\S_j \, f\S \mleft( Y\S_j \mright). \label{eq:GARK:step}
    \end{align}
\end{subequations}
The coefficients of these methods can be organized into the following Butcher tableau:
\begin{equation} \label{eq:tableau}
    \begin{butchertableau}{c|c}
        \mathbf{A}\FF & \mathbf{A}\FS \\ \hline
        \mathbf{A}\SF & \mathbf{A}\SS \\ \hline
        \mathbf{b}\F* & \mathbf{b}\S*
    \end{butchertableau}.
\end{equation}
The fast method $\left( \mathbf{A}\FF, \mathbf{b}\F, \mathbf{c}\F \right)$ has $\mathbf{s}\F$ stages, and the slow method $\left( \mathbf{A}\SS, \mathbf{b}\S, \mathbf{c}\S \right)$, has $\mathbf{s}\S$ stages.  Also, we use the notation
\begin{equation*}
    \mathbf{A} = \begin{bmatrix}
        \mathbf{A}\FF & \mathbf{A}\FS \\
        \mathbf{A}\SF & \mathbf{A}\SS
    \end{bmatrix}, \quad
    \mathbf{b} = \begin{bmatrix}
        \mathbf{b}\F \\ \mathbf{b}\S
    \end{bmatrix}, \quad
    \mathbf{s} = \mathbf{s}\F + \mathbf{s}\S.
\end{equation*}
A common simplifying assumption, which ensures the fast and slow functions in \cref{eq:GARK} are computed at consistent times, is internal consistency:
\begin{equation}
    \mathbf{c}\F \coloneqq \mathbf{A}\FF \, \vecone_{s\F} = \mathbf{A}\FS \, \vecone_{s\S}
    \quad \text{and} \quad
    \mathbf{c}\S \coloneqq \mathbf{A}\SS \, \vecone_{s\S} = \mathbf{A}\SF \, \vecone_{s\F}.
\end{equation}

In~\cite{gunther2016multirate}, it was shown how several types of multirate Runge--Kutta methods can be described as GARK methods.  In one step of a multirate method, the slow dynamics $f\S$ are integrated with a macro-step of $H$, and the fast dynamics $f\F$ are integrated with a micro-step of $h = H / M$.  The multirate ratio $M$ is a positive integer.  Information between the two partitions is shared via the coupling matrices $\mathbf{A}\FS$ and $\mathbf{A}\SF$.  In this section, we present two families of multirate Runge--Kutta methods viewed as special cases of the GARK framework.

\subsection{Standard MrGARK} \label{sec:MrGARK:standard}

A standard MrGARK method is built on an $s\F$-stage fast base method $\left( A\FF, b\F, c\F \right)$ and an $s\S$-stage slow base method $\left( A\SS, b\S, c\S \right)$.  From \cite{gunther2016multirate}, one step proceeds as
\begin{subequations} \label{eq:MrGARK}
    \begin{align}
        \begin{split}
        Y_i\S &= y_n + H \, \sum_{j = 1}^{s\S} a_{i,j}\SS \, f\S \mleft( Y_j\S \mright) + h \, \sum_{\lambda = 1}^M \sum_{j = 1}^{s\F} a_{i,j}\SFL \, f\F \mleft( Y_j \FL \mright),\\
        & \quad \text{for } i = 1, \ldots, s\S, 
        \end{split}\\
        \begin{split}
            Y_i\FL &= \widetilde{y}_{n + (\lambda - 1)/M} + H \, \sum_{j = 1}^{s\S} a\FSL_{i,j} \, f\S \mleft( Y_j\S \mright) + h \, \sum_{j = 1}^{s\F} a\FF_{i,j} \, f\F \mleft( Y_j\FL \mright), \\
            & \quad  \text{for } i = 1, \ldots, s\F, \\
            \widetilde{y}_{n + \lambda/M} &= \widetilde{y}_{n + (\lambda - 1)/M} + h \, \sum_{i = 1}^{s\F} b\F_i \, f\F \mleft( Y_i\FL \mright), \\
            & \quad \text{for } \lambda = 1, \ldots, M,
        \end{split} \\
        y_{n+1} &= \widetilde{y}_{n + M/M} + H \, \sum_{i = 1}^{s\S} b\S_i \, f\S \mleft( Y_i\S \mright),
    \end{align}
\end{subequations}
where the micro-steps start with $\widetilde{y}_n = y_n$.  The corresponding Butcher tableau for \cref{eq:MrGARK} is
\begin{equation} \label{eq:MrGARK_tableau}
    \begin{butchertableau}{c|c}
        \mathbf{A}\FF & \mathbf{A}\FS \\ \hline
        \mathbf{A}\SF & \mathbf{A}\SS \\ \hline
        \mathbf{b}\F* & \mathbf{b}\S*
    \end{butchertableau}
    \coloneqq
    \begin{butchertableau}{ccc|c}
        \frac{1}{M} A\FF & \cdots & 0 & A\FSL[1] \\
        \vdots & \ddots & \vdots & \vdots \\
        \frac{1}{M} \vecone_{s\F} b\F* & \cdots & \frac{1}{M} A\FF & A\FSL[M] \\ \hline
        \frac{1}{M} A\SFL[1] & \cdots & \frac{1}{M} A\SFL[M] & A\SS \\ \hline
        \frac{1}{M} b\F* & \cdots & \frac{1}{M} b\F* & b\S*
    \end{butchertableau}.
\end{equation}
Note that $\mathbf{s}\F = M s\F$ and $\mathbf{s}\S = s\S$.

If the fast and slow base methods are identical, the method is called \textit{telescopic} as it can be applied in a nested fashion to more than two partitions \cite{gunther2016multirate}.  Further, MrGARK methods can be classified as coupled or decoupled \cite{sarshar2019design}.  Decoupled methods only have implicitness in the base methods; the stages used in coupling can always be computed before they are needed.  Coupled methods, on the other hand, have fast and slow stages that are implicitly defined in terms of each other and that must be computed together.  Decoupled methods can be implemented more efficiently, but can sacrifice stability as we will see in \cref{sec:stability}.

Order conditions for this family of methods comes from applying the particular multirate structure of \cref{eq:MrGARK_tableau} into the GARK order conditions.  The conditions up to order four are provided in \cite{sarshar2019design}.  A similar approach has been use in \cite{sandu2018class,roberts2018coupled} to derive order conditions for MRI-GARK methods and in \cite{Sandu_2015_GARK,sexton2018relaxed} for multirate infinitesimal step methods \cite{Knoth_1998_MR-IMEX,Schlegel_2009_RFSMR,Wensch_2009_MIS}.
\todo{Added earlier MIS citations}

\subsection{\cmrgark*{} MrGARK}
\label{sec:MrGARK:cs}

Another multirate strategy, based on the early work of Rice \cite{rice1960split} and the later developments in \cite{savcenco2007multirate,verhoeven2006general}, is the \cmrgark{} approach.  The idea is to first take a macro-step of the full system \cref{eq:ode} called the compound step.  Over the large timestep, the fast integration is inaccurate and discarded.  The fast partition is then reintegrated using a smaller timestep.  Slow coupling information is required at the intermediate micro-steps and can come from an interpolant of the compound step solution.  Note the fast partition is integrated twice for each timestep, but no extrapolation is required for the coupling.  Moreover, an error estimate from the compound step, say from an embedded method, can be used to dynamically determine at each step which variables exceed accuracy tolerances and should form the fast components \cite{savcenco2007multirate}.

Traditionally, \cmrgark{} methods have been posed for component partitioned systems \cref{eq:ode_comp}, however, they easily extend to additively partitioned systems \cref{eq:ode}.  One step of a \cmrgark{} MrGARK scheme is given by
\begin{subequations}
    \label{eq:MrGARK_CS}
    \begin{align}
        \label{eq:MrGARK_CS:macro-step}
        Y_i &= y_n + H \, \sum_{j=1}^s a_{i,j} \, f(Y_j), \qquad i = 1, \ldots, s, \\
        \begin{split} \label{eq:MrGARK_CS:micro-step}
            Y_i\FL &= \widetilde{y}_{n + (\lambda - 1)/M} + h \sum_{j = 1}^{s\F} a\FF_{i,j} \, f\F \mleft( Y_j\FL \mright) + H \, \sum_{j = 1}^{s\S} a\FSL_{i,j} \, f\S \mleft( Y_j \mright),\\
            & \quad \text{for } i = 1, \ldots, s\F, \\
            \widetilde{y}_{n + \lambda/M} &= \widetilde{y}_{n + (\lambda - 1)/M} + h \, \sum_{i = 1}^{s\F} b\F_i \, f\F \mleft( Y_i\FL \mright),\\
            & \quad \text{for } \lambda = 1, \ldots, M,
        \end{split} \\
        y_{n+1} &= \widetilde{y}_{n + M/M} + H \, \sum_{i = 1}^{s\S} b\S_i \, f\S \mleft( Y_i \mright),
    \end{align}
\end{subequations}
where the micro-steps start with $\widetilde{y}_n = y_n$.  The corresponding tableau is
\begin{equation*}
\begin{butchertableau}{c|c}
            \mathbf{A}\FF & \mathbf{A}\FS \\ \hline
            \mathbf{A}\SF & \mathbf{A}\SS \\ \hline
            \mathbf{b}\F* & \mathbf{b}\S*
    \end{butchertableau}
    \coloneqq
    \begin{butchertableau}{cccc|c}
        A & 0 & \cdots & 0 & A \\
        0  & \frac{1}{M} A & \cdots & 0 & A\FSL[1] \\
        0 & \vdots & \ddots & \vdots & \vdots \\
        0 & \frac{1}{M} \vecone_s b^T & \cdots & \frac{1}{M} A & A\FSL[M] \\ \hline
        A & 0 & \cdots & 0 & A \\ \hline
        0 & \frac{1}{M} b^T & \cdots & \frac{1}{M} b^T & b^T
    \end{butchertableau}.
\end{equation*}
This family of methods is telescopic, coupled in the macro-step \cref{eq:MrGARK_CS:macro-step}, and decoupled in the remaining fast micro-steps \cref{eq:MrGARK_CS:micro-step}.  The coupling matrix $A\FSL$ can be interpreted as the interpolation weights for the slow tendencies.  These multirate methods will preserve the order of the base method if the interpolant is sufficiently accurate.  We note that this is a sufficient condition, but not always necessary.  The GARK order conditions can be used to derive precise conditions to achieve a particular order.

\begin{theorem}[\cmrgark*{} MrGARK order conditions]
    An internally consistent \cmrgark{} MrGARK method has order four if and only if the base method $(A, b, c)$ has order four and the following coupling conditions hold:
    \begin{subequations} \label{eq:cs_order_conditions}
        \begin{align}
             &M \, A\FSL \, \vecone_{s\S}  = (\lambda - 1 ) \, \vecone_{s\F} + c,  & \text{(int. consistency)} \label{eq:CS-OC-a} \\
            &\frac{M}{6} = \sum_{\lambda=1}^M b^T \, A\FSL \, c, & \text{(order 3)} \label{eq:CS-OC-b} \\
            &\frac{M^2}{8} = \sum_{\lambda=1}^M (\lambda-1) \, b^T \, A\FSL \, c + \sum_{\lambda=1}^M (b \times c)^T \, A\FSL \, c, & \text{(order 4)} \label{eq:CS-OC-c}\\
            &\frac{M}{12} = \sum_{\lambda=1}^M b^T \, A\FSL \, c^2, & \text{(order 4)} \label{eq:CS-OC-d}\\
            &\frac{M^2}{24} = \sum_{\lambda=1}^M b^T \, A \, A\FSL \, c + \sum_{\lambda=1}^M (M-\lambda) \, b^T \, A\FSL \, c, & \text{(order 4)} \label{eq:CS-OC-e}\\
            &\frac{M}{24} = \sum_{\lambda=1}^M b^T \, A\FSL \, A \, c. & \text{(order 4)} \label{eq:CS-OC-f} 
        \end{align}
    \end{subequations}
\end{theorem}

\begin{proof}
    From~\cite{Sandu_2015_GARK}, an internally consistent GARK method has order four if and only the base methods have order four and the following coupling conditions hold.
    \par{Condition 3a:}
	\begin{equation*}
		\frac{1}{6}
		= \mathbf{b}\F* \, \mathbf{A}\FS \, \mathbf{c}\S
		= \frac{1}{M} \, \sum_{\lambda=1}^M b^T \, A\FSL \, c.
	\end{equation*}
	\par{Condition 3b:}
	\begin{equation*}
		\frac{1}{6}
		= \mathbf{b}\S* \, \mathbf{A}\SF \, \mathbf{c}\F
		= b^T \, A \, c.
	\end{equation*}
	\par{Condition 4a:}
	\begin{align*}
		\frac{1}{8}
		&= \left( \mathbf{b}\F \times \mathbf{c}\F \right)^T \mathbf{A}\FS \, \mathbf{c}\S \\
		&= \frac{1}{M^2} \, \sum_{\lambda=1}^M (b \times (c + (\lambda - 1) \, \vecone_s))^T \, A\FSL \, c\S. \\
		&= \frac{1}{M^2} \, \sum_{\lambda=1}^M (\lambda-1) \, b^T \, A\FSL \, c + \frac{1}{M^2} \sum_{\lambda=1}^M (b \times c)^T \, A\FSL \, c
	\end{align*}
	\par{Condition 4b:}
	\begin{equation*}
	    \frac{1}{8}
		= \left( \mathbf{b}\S \times \mathbf{c}\S \right)^T \mathbf{A}\SF \, \mathbf{c}\F
		= (b \times c)^T \, A \, c.
	\end{equation*}
	\par{Condition 4c:}
	\begin{equation*}
		\frac{1}{12}
		= \mathbf{b}\F* \, \mathbf{A}\FS \, \mathbf{c}\S[\times 2]
		= \frac{1}{M} \, \sum_{\lambda=1}^M b^T \, A\FSL \, c^2.
	\end{equation*}
	\par{Condition 4d:}
	\begin{equation*}
		\frac{1}{12}
		= \mathbf{b}\S* \, \mathbf{A}\SF \, \mathbf{c}\F[\times 2]
		= b^T \, A \, c^2.
	\end{equation*}
	\par{Condition 4e:}
	\begin{align*}
		\frac{1}{24}
		&= \mathbf{b}\F* \, \mathbf{A}\FF \, \mathbf{A}\FS \, \mathbf{c}\S \\
		&= \frac{1}{M^2} \, \sum_{\lambda=1}^M \left( b^T \, A + \sum_{k=1}^{M-\lambda} b^T \, \vecone_s \, b^T \right) A\FSL \, c \\
		&= \frac{1}{M^2} \, \sum_{\lambda=1}^M b^T \, A \, A\FSL \, c + \frac{1}{M^2} \sum_{\lambda=1}^M (M-\lambda) \, b^T \, A\FSL \, c.
	\end{align*}
	\par{Condition 4f:}
	\begin{equation*}
		\frac{1}{24}
		= \mathbf{b}\F* \, \mathbf{A}\FS \, \mathbf{A}\SF \, \mathbf{c}\F
		= \frac{1}{M} \, \sum_{\lambda=1}^M b^T \, A\FSL \, A \, c.
	\end{equation*}
	\par{Condition 4g:}
	\begin{equation*}
		\frac{1}{24}
		= \mathbf{b}\F* \, \mathbf{A}\FS \, \mathbf{A}\SS \, \mathbf{c}\S
		= \frac{1}{M} \, \sum_{\lambda=1}^M b^T \, A\FSL \, A \, c.
	\end{equation*}
	\par{Condition 4h:}
	\begin{equation*}
		\frac{1}{24}
		= \mathbf{b}\S* \, \mathbf{A}\SS \, \mathbf{A}\SF \, \mathbf{c}\F
		= b^T \, A \, A \, c.
	\end{equation*}
	\par{Condition 4i:}
	\begin{equation*}
		\frac{1}{24}
		= \mathbf{b}\S* \, \mathbf{A}\SF \, \mathbf{A}\FS \, \mathbf{c}\S
		= b^T \, A \, A \, c.
	\end{equation*}
	\par{Condition 4j:}
	\begin{equation*}
		\frac{1}{24}
		= \mathbf{b}\S* \, \mathbf{A}\SF \, \mathbf{A}\FF \, \mathbf{c}\F
		= b^T \, A \, A \, c.
	\end{equation*}
	
	Note that conditions 3b, 4b, 4d, and 4h--j resolve to order conditions of the base method, and thus, are satisfied if and only if the base method has order four.  Further, condition 4g is identical to 4f.  The remaining order conditions give \cref{eq:cs_order_conditions}.
	\qed
\end{proof}

\section{Multirate Linear Stability Analysis}
\label{sec:stability}

In the analysis of single rate Runge--Kutta methods, it is common to apply methods to the Dahlquist test problem
\begin{equation} \label{eq:Dahlquist}
    y' = \lambda \, y,
\end{equation}
with $\lambda \in \C^- = \{ z \in \C : \Re(z) \le 0 \}$.  This yields the well-known linear stability function
\begin{equation} 
\label{eq:stability}
    R(z) = 1 + z \, b^T \, (I - z \, A)^{-1} \, \vecone_s,
\end{equation}
where $z = \lambda \, h$.  It suffices to only examine a scalar problem \cref{eq:Dahlquist} because in the case $\lambda$ and $z$ are matrices, the behavior of $R(z)^m$ as $m \to \infty$ only depends on the scalar eigenvalues of $z$.  That is, the choice of basis for a system of linear ODEs does not affect a Runge--Kutta method's stability.

As noted by Gear \cite{gear1984multirate}, this property does not hold for multirate and other partitioned schemes.  For this reason, the linear stability analysis becomes significantly more complex.  In this section, we analyze and compare linear stability for both scalar and two-dimensional (2D) test problems.  The scalar test problem is a simple model of an additively partitioned system  \eqref{eq:ode} where the Jacobians of the two processes triangularize simultaneously. The 2D problem is a simple model for a component partitioned system \eqref{eq:ode_comp}, where each component's dynamics as well as the interaction between components are linear.

In this section, we will focus on two-partitioned GARK methods for simplicity.  Nearly all of the stability analysis, however, has straightforward generalizations to the full $N$-partitioned GARK framework.

\subsection{Scalar Test Problem}

The simplest generalization of \cref{eq:Dahlquist} for two-partitioned multirate methods is the scalar test problem
\begin{equation} \label{eq:scalar_test}
    y' = \lambda\F \, y + \lambda\S \, y,
\end{equation}
where, $\lambda\F, \lambda\S \in \C^-$.  As shown in \cite{Sandu_2015_GARK}, when \cref{eq:MrGARK} is applied to \cref{eq:scalar_test}, we arrive at the stability function
\begin{equation} \label{eq:scalar_stability}
    R_1 \mleft( z\F, z\S \mright) = 1 + \mathbf{b}^T \, Z \, (\eye{\mathbf{s}} - \mathbf{A} \, Z)^{-1} \, \vecone_{\mathbf{s}},
\end{equation}
where $z\F = H \, \lambda\F$, $z\S = H \, \lambda\S$, and
\begin{equation*}
    Z = \begin{bmatrix}
    z\F \, \eye{\mathbf{s}\F} & 0 \\
    0 & z\S \, \eye{\mathbf{s}\S}
    \end{bmatrix}.
\end{equation*}
\begin{definition}[Scalar region of absolute stability]
    The set 
    $$S_1 = \left\{ (z\F, z\S) \in \C \times \C : \abs{R_1 \mleft( z\F, z\S \mright)} \leq 1 \right\}$$ is the region of absolute stability for the test problem \cref{eq:scalar_test}.  A GARK method is called scalar A-stable if $S_1 \supseteq \C^- \times \C^-$.  Further, a GARK method is called scalar L-stable if it is scalar A-stable,
    \begin{equation} \label{eq:scalar_rinf}
        \lim_{z\F \rightarrow \infty} R_1 \mleft( z\F, z\S \mright) = 0,
        \quad \text{and} \quad
        \lim_{z\S \rightarrow \infty} R_1 \mleft( z\F, z\S \mright) = 0.
    \end{equation}
\end{definition}

\begin{definition}[{Scalar A$(\alpha)$- and L$(\alpha)$-stability}] \label{def:scalar_stability_alpha}
    A GARK method is scalar A$(\alpha)$-stable if $S_1 \supseteq W(\alpha) \times W(\alpha)$, where $W(\alpha)$ is the wedge $\left\{ z \in \C : \abs{\arg{(-z)}} < \alpha, z \ne 0 \right\}$.  A scalar A$(\alpha)$-stable GARK method that additionally satisfies \cref{eq:scalar_rinf} is called scalar L$(\alpha)$-stable.
\end{definition}

One way to determine if a single rate Runge--Kutta method is stable in the entire left half-plane is by ensuring stability on the imaginary axis and that the poles of $R(z)$ are in the right half-plane \cite[Section IV.3]{Hairer_book_II}.  Further, stability on the imaginary axis is equivalent to the E-polynomial
\begin{equation*}
    E(y) = Q(i \, y) \, Q(-i \, y) - P(i \, y) \, P(-i \, y)
\end{equation*}
being nonnegative for all $y \in \R$.  Here, $P$ and $Q$ are the numerator and denominator of \cref{eq:stability}, respectively.  As we will now show, these practical techniques for determining linear stability have simple and direct generalizations for GARK methods applied to \cref{eq:scalar_test}.

\begin{theorem}[Necessary and sufficient condition for scalar A-stability] \label{thm:R1_imag}
    The GARK method \cref{eq:GARK} is scalar A-stable if and only if
    \begin{equation}
        \abs{R_1 \mleft( i \, y\F, i \, y\S \mright)} \leq 1 \qquad \text{for all } y\F, y\S \in \R
    \end{equation}
    and $R_1$ is analytic over $\C^- \times \C^-$.
\end{theorem}

\begin{proof}
    This follows from the multivariate maximum principle (see for example \cite{scheidemann2005introduction}).
    \qed

\end{proof}

\begin{remark}[Finding $A(\alpha)$-stability regions] 
    The maximum principle can also be used to efficiently determine the angle for scalar A$(\alpha)$-stability.  Instead of ensuring stability for all points inside a 4D wedge $W(\alpha) \times W(\alpha)$, one can limit the analysis to the boundary points $\partial W(\alpha) \times \partial W(\alpha)$.
\end{remark}

Notably, \Cref{thm:R1_imag} reduces the space on which we have to check for A-stability from four to two dimensions.  For multirate methods, however, $R_1$ is different for each value of $M$, thus adding another dimension to consider.

\begin{theorem}[E-polynomial]
    The E-polynomial for GARK methods is
    \begin{align*}
        E_1 \mleft( y\F, y\S \mright)
        &= Q_1 \mleft( i \, y\F, i \, y\S \mright) \, Q_1 \mleft( -i \, y\F, -i \, y\S \mright) \\
        &\quad - P_1 \mleft( i \, y\F, i \, y\S \mright) \, P_1 \mleft( -i \, y\F, -i \, y\S \mright),
    \end{align*}
    where $P_1$ and $Q_1$ are the numerator and denominator of \cref{eq:scalar_stability}, respectively.  The scalar stability region of a method contains the imaginary axes if and only if the E-polynomial is nonnegative for all $y\F, y\S \in \R$.
\end{theorem}

\begin{proof}
    Following the single rate approach presented in \cite[Section IV.3]{Hairer_book_II}, we have that
    \begin{align*}
        1 &\ge \abs{R_1 \mleft( i \, y\F, i \, y\S \mright)}^2 \\
        0 &\le \abs{Q_1 \mleft( i \, y\F, i \, y\S \mright)}^2 - \abs{P_1 \mleft( i \, y\F, i \, y\S \mright)}^2 \\
        0 &\le Q_1 \mleft( i \, y\F, i \, y\S \mright) \, \overline{Q_1 \mleft( i \, y\F, i \, y\S \mright)} - P_1 \mleft( i \, y\F, i \, y\S \mright) \, \overline{P_1 \mleft( i \, y\F, i \, y\S \mright)} \\
        0 &\le Q_1 \mleft( i \, y\F, i \, y\S \mright) \, Q_1 \mleft( -i \, y\F, -i \, y\S \mright) - P_1 \mleft( i \, y\F, i \, y\S \mright) \, P_1 \mleft( -i \, y\F, -i \, y\S \mright) \\
        0 &\le E_1 \mleft( y\F, y\S \mright).
    \end{align*}
    Since each of these inequalities is equivalent, the statement is proven.
    \qed
\end{proof}

\subsection{2D Test Problem}

Another test problem, first proposed in \cite{gear1974multirate}, and later used in \cite{kvaerno2000stability,savcenco2008comparison,hundsdorfer2009analysis,constantinescu2013extrapolated}, is the 2D linear test problem
\begin{equation} \label{eq:2d_test}
    \begin{bmatrix}
        y\F \\ y\S
    \end{bmatrix}' =
    \underbrace{
        \begin{bmatrix}
            \lambda\F & \eta\S \\
            \eta\F & \lambda\S
        \end{bmatrix}
    }_{\Lambda}
    \begin{bmatrix}
        y\F \\ y\S
    \end{bmatrix}.
\end{equation}
Here, the exact solution must be bounded.  That is, the eigenvalues of $\Lambda$ have nonpositive real parts and eigenvalues on the imaginary axis are regular.  Further, we enforce that $\lambda\F, \lambda\S \in \C^-$ so the individual partitions have bounded dynamics.  We will denote the set of these special exponentially bounded matrices by $\mathbb{M}$, and this test problem will be referred to as the complex 2D test problem.  Many authors have considered simplifying assumptions including restricting $\Lambda$ to real entries.  In this case, the constraints on $\Lambda$ simplify to $\lambda\F, \lambda\S \leq 0$, $\eta\F \eta\S \leq \lambda\F \lambda\S$, and a zero eigenvalue must be regular.  We will refer to this problem as the real 2D test problem.

When \cref{eq:GARK} is applied to \cref{eq:2d_test}, we arrive at the stability matrix
\begin{align} \label{eq:2d_stability}
    \begin{split}
        & \quad R_2 \mleft(
        \begin{bmatrix}
            z\F & w\S \\
            w\F & z\S
        \end{bmatrix}
        \mright) \\
        &= \eye{2} + \begin{bmatrix}
            \mathbf{b}\F* & 0 \\
            0 & \mathbf{b}\S*
        \end{bmatrix}
        \begin{bmatrix}
            \eye{\mathbf{s}\F} - z\F \, \mathbf{A}\FF & -w\S \, \mathbf{A}\FS \\
            -w\F \, \mathbf{A}\SF & \eye{\mathbf{s}\S} - z\S \, \mathbf{A}\SS
        \end{bmatrix}^{-1} \\
        & \quad
        \begin{bmatrix}
            z\F \, \vecone_{\mathbf{s}\F} & w\S \, \vecone_{\mathbf{s}\F} \\
            w\F \, \vecone_{\mathbf{s}\S} & z\S \, \vecone_{\mathbf{s}\S}
        \end{bmatrix},
    \end{split}
\end{align}
where
\begin{equation*}
    \begin{bmatrix}
        z\F & w\S \\
        w\F & z\S
    \end{bmatrix} =
    H \begin{bmatrix}
        \lambda\F & \eta\S \\
        \eta\F & \lambda\S
    \end{bmatrix}.
\end{equation*}
\begin{definition}[Complex 2D region of absolute stability]
    The set $$S_2 = \left\{ Z \in \C^{2 \times 2} : R_2(Z) \text{ power bounded} \right\}$$ is the complex 2D region of absolute stability for the test problem \cref{eq:2d_test}.  A GARK method is called complex 2D A-stable if $S_2 \supseteq \mathbb{M}$.
\end{definition}

\begin{definition}[Real 2D region of absolute stability]
    The set $$\widehat{S}_2 = \left\{ Z \in \R^{2 \times 2} : R_2(Z) \text{ power bounded} \right\}$$ is the real 2D region of absolute stability for the test problem \cref{eq:2d_test}.  A GARK method is called real 2D A-stable if $\widehat{S}_2 \supseteq (\mathbb{M} \cap \R^{2 \times 2})$.
\end{definition}

For both cases of the 2D test problem, the power boundedness condition makes finding necessary and sufficient conditions for stability significantly more challenging.
Considering test problems on the boundary of $\mathbb{M}$ does provide important necessary conditions. Consider the particular test problem
\begin{equation} \label{eq:2d_imag}
    y' = \begin{bmatrix}
        0 & \eta \\
        -\eta & 0
    \end{bmatrix} y,
\end{equation}
which has purely imaginary eigenvalues for $\eta \in \R$.  
Note that
\begin{equation} \label{eq:2d_stability_imag}
    \begin{split}
        R_2 \mleft(
        \begin{bmatrix}
            0 & w \\
            -w & 0
        \end{bmatrix}
        \mright) &= \eye{2} + w
        \begin{bmatrix}
            \mathbf{b}\F* & 0 \\
            0 & \mathbf{b}\S*
        \end{bmatrix}
        \begin{bmatrix}
            \eye{\mathbf{s}\F} & -w \, \mathbf{A}\FS \\
            w \, \mathbf{A}\SF & \eye{\mathbf{s}\S}
        \end{bmatrix}^{-1}
        \begin{bmatrix}
            0 & \vecone_{\mathbf{s}\F} \\
            -\vecone_{\mathbf{s}\S} & 0
        \end{bmatrix} \\
        &= \begin{bmatrix}
            1 - w^2 \, \mathbf{b}\F* \, \mathbf{A}\FS \, \mathbf{d}\S &
            w \, \mathbf{b}\F* \, \mathbf{d}\F \\
            -w \, \mathbf{b}\S* \, \mathbf{d}\S &
            1 - w^2 \, \mathbf{b}\S* \, \mathbf{A}\SF \, \mathbf{d}\F
        \end{bmatrix},
    \end{split}
\end{equation}
where $w = H \, \eta$ and
\begin{align*}
    \mathbf{d}\F &= \left( \eye{\mathbf{s}\F} + w^2 \, \mathbf{A}\FS \, \mathbf{A}\SF \right)^{-1} \vecone_{\mathbf{s}\F}, \\
    \mathbf{d}\S &= \left( \eye{\mathbf{s}\S} + w^2 \, \mathbf{A}\SF \, \mathbf{A}\FS \right)^{-1} \vecone_{\mathbf{s}\S}.
\end{align*}

An important property of this stability function, which will be used later for \cref{thm:decoupled}, is that it depends on the coupling coefficients but not the base method coefficients $A\FF$ and $A\SS$.

\begin{remark}[Other test problems]
The 2D problem can be generalized to the linear block system
\begin{equation}
    \label{eq:block_stability}
    \begin{bmatrix}
        y\F \\ y\S
    \end{bmatrix}' =
    \begin{bmatrix}
        \Lambda\F & E\S \\
        E\F & \Lambda\S
    \end{bmatrix}
    \begin{bmatrix}
        y\F \\ y\S
    \end{bmatrix}.
\end{equation}
This problem has been considered in~\cite{andrus1993stability,gear1984multirate}.  An even more general block system was used by Skelboe in~\cite{skelboe1989stability}.  We do not consider these block generalizations further as we find that the 2D problem already poses a surprisingly challenging test problem.
\end{remark}

\subsection{Comparison of Stability Test Problems}

When designing an implicit method, unconditional stability is a highly desirable property. A natural question is which test problem should be used to determine stability.  In this section, we explore the relationships among the different stability criteria in order to address this question. Consider, for example, the GARK method given by the tableau below:
\begin{equation*}
    \begin{butchertableau}{c|c}
        1 & 0 \\ \hline
        1 & 1 \\ \hline
        1 & 1
    \end{butchertableau}.
\end{equation*}
This method is scalar L-stable and even algebraically stable \cite{Sandu_2015_GARK}, but
\begin{equation*}
    \rho \mleft( R_2 \mleft(
    \begin{bmatrix}
        -1 & 1 \\
        -10 & -1
    \end{bmatrix}
    \mright) \mright) =
    \frac{\sqrt{5}+3}{4} > 1,
\end{equation*}
with $\rho$ the spectral radius operator.  Thus, it is only conditionally stable for the real and complex 2D test problems.

Conversely, consider the GARK method
\begin{equation*}
    \begin{butchertableau}{cc|cc}
        \frac{1}{4} & 0 & \frac{1}{4} & \frac{1}{4} \\
        \frac{1}{4} & \frac{1}{4} & \frac{1}{4} & \frac{1}{4} \\ \hline
        \frac{1}{4} & \frac{1}{4} & \frac{1}{4} & 0 \\
        \frac{1}{4} & \frac{1}{4} & \frac{1}{4} & \frac{1}{4} \\ \hline
        \frac{1}{2} & \frac{1}{2} & \frac{1}{2} & \frac{1}{2}
    \end{butchertableau}.
\end{equation*}
The base method is only A(\ang{45}) stable, and thus, it is easy to show the GARK method is conditionally stable with respect to the scalar test problem:
\begin{equation}
    \abs{R_1(-4 + 8 i, 0)} = \frac{\sqrt{17}}{4} > 1.
\end{equation}
For the real 2D test problem, this GARK method is A-stable.  This result reveals a shortcoming of the real 2D test problem: the individual partitions have purely real eigenvalues.  Ideally, a test problem should reveal instabilities of the base methods off the real axis.  Despite the apparent independence of the stability functions \cref{eq:scalar_stability,eq:2d_stability}, we do note that
\begin{subequations} \label{eq:test_eqalities}
    \begin{align}
        R_1 \mleft( z\F, z\S \mright) &=
        \begin{bmatrix}
            z\F & z\S
        \end{bmatrix}
        R_2 \mleft( \begin{bmatrix}
            z\F & z\S \\
            z\F & z\S \\
        \end{bmatrix} \mright)
        \bgroup
        \renewcommand{\arraystretch}{1.25}
        \begin{bmatrix}
            \frac{\alpha}{z\F} \\ \frac{1 - \alpha}{z\S}
        \end{bmatrix}
        \egroup, \label{eq:test_eqalities_1} \\
        %
        &=
        \begin{bmatrix}
            1 & 1
        \end{bmatrix}
        R_2 \mleft( \begin{bmatrix}
            z\F & z\F \\
            z\S & z\S \\
        \end{bmatrix} \mright)
        \begin{bmatrix}
            \alpha \\ 1 - \alpha
        \end{bmatrix}, \label{eq:test_eqalities_2}
    \end{align}
\end{subequations}
for any $\alpha\in \C$.

When \cref{eq:2d_test} is taken to have complex entries, however, there is a meaningful connection to the scalar test problem.
\begin{theorem} \label{thm:a-stability}
    If a GARK method is A-stable with respect to the complex 2D test problem, then it is A-stable with respect to the scalar test problem.
\end{theorem}
\begin{proof}
    First, we define
    \begin{equation}
        R_2 \mleft( \begin{bmatrix}
            z\F & z\F \\
            z\S & z\S \\
        \end{bmatrix} \mright)
        = \begin{bmatrix}
            r_{1,1} & r_{1,2} \\
            r_{2,1} & r_{2,2}
        \end{bmatrix}.
    \end{equation}
    Since \cref{eq:test_eqalities_2} must hold for all $\alpha$,
    \begin{equation*}
        \text{const} =
        \begin{bmatrix} 1 & 1 \end{bmatrix}
        \begin{bmatrix}
            r_{1,1} & r_{1,2} \\
            r_{2,1} & r_{2,2}
        \end{bmatrix}
        \begin{bmatrix}
            \alpha \\ 1 - \alpha
        \end{bmatrix} \\
        = \alpha \, (r_{1,1} + r_{2,1} - r_{1,2} - r_{2,2}) + r_{1,2} + r_{2,2}.
    \end{equation*}
    Thus, $r_{1,1} + r_{2,1} - r_{1,2} - r_{2,2} = 0$ and
    \begin{equation}
        R_2 \mleft( \begin{bmatrix}
            z\F & z\F \\
            z\S & z\S \\
        \end{bmatrix} \mright)
        = \begin{bmatrix}
            r_{1,1} & r_{1,2} \\
            r_{2,1} & r_{1,1} + r_{2,1} - r_{1,2}
        \end{bmatrix}.
    \end{equation}
    
    Due to this structure, $r_{1,1} + r_{2,1}$ is an eigenvalue, and if a GARK method is A-stable for the 2D test problem, then $\abs{r_{1,1} + r_{2,1}} \leq 1$.  Using \cref{eq:test_eqalities_2} with $\alpha=1$, we have that
    \begin{align*}
        R_1 \mleft( z\F, z\S \mright) &=
        \begin{bmatrix}
            1 & 1
        \end{bmatrix}
        R_2 \mleft( \begin{bmatrix}
            z\F & z\F \\
            z\S & z\S \\
        \end{bmatrix} \mright)
        \begin{bmatrix}
            1 \\ 0
        \end{bmatrix} \\
        &= r_{1,1} + r_{2,1} \\
        \abs{R_1 \mleft( z\F, z\S \mright)} &\leq 1.
    \end{align*}
    Thus, the method is A-stable for the scalar test problem.
    \qed
\end{proof}

While the 2D test problem may be a more thorough, reliable, and informative method of assessing stability, it is also more difficult to analyze and visualize due to the high-dimensional space of test problems.  We summarize the hierarchy of linear stability properties in \Cref{fig:stability}.

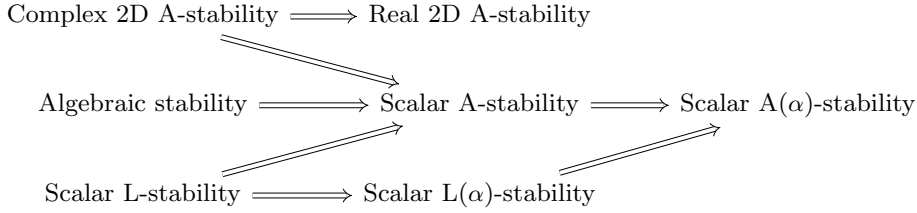
\begin{figure}[ht]
    \centering
    \begin{tikzcd}[arrows=Rightarrow]
       \text{Complex 2D A-stability} \arrow{r}{} \arrow{dr}{} & \text{Real 2D A-stability} \\
       \text{Algebraic stability} \arrow{r}{} & \text{Scalar A-stability} \arrow{r}{} & \text{Scalar A$(\alpha)$-stability} \\
       \text{Scalar L-stability} \arrow{r}{} \arrow{ru}{} & \text{Scalar L$(\alpha)$-stability} \arrow{ru}{}
    \end{tikzcd}
    \caption{Stability implications for the various linear test problems. In general, no implication arrows are reversible.}
    \label{fig:stability}
\end{figure}

\begin{lemma} \label{lem:nilpotent}
    For a decoupled GARK method, the following matrix is nilpotent:
    \begin{equation*}
        \begin{bmatrix}
            0 & \mathbf{A}\FS \\
            \mathbf{A}\SF & 0
        \end{bmatrix}.
    \end{equation*}
\end{lemma}

\begin{proof}
    The full matrix $\mathbf{A}$ can be viewed as the adjacency matrix of a weighted directed graph.  Cycles indicate the method is implicit, and by the definition of a decoupled method, implicitness only comes from the base methods.  With the base method coefficients set to zero, the directed graph becomes acyclic: a property equivalent to nilpotency of the adjacency matrix.
    \qed
\end{proof}




\begin{theorem} \label{thm:decoupled}
    A decoupled GARK method consistent with \cref{eq:ode} (first order accurate) cannot be A-stable for the real 2D test problem.
\end{theorem}

\begin{proof}
    Consider the particular test problem given in \cref{eq:2d_imag}.  Note that in \cref{eq:2d_stability_imag}, the matrix being inverted is the sum of an identity matrix and a nilpotent matrix by the decoupled assumption and \cref{lem:nilpotent}.  Expanding the inverse in a Neumann series reveals $\mathbf{d}\F$ and $\mathbf{d}\S$ must be even polynomials in $w$ of finite degree.  Moreover, the off-diagonal terms of the stability matrix satisfy
    \begin{alignat*}{2}
        w \, \mathbf{b}\F* \, \mathbf{d}\F
        &= w \, \mathbf{b}\F* \left( \eye + w^2 \, \mathbf{A}\FS \, \mathbf{A}\SF + \ldots \right) \vecone_{\mathbf{s}\F}
        &&= w + w^3 \, p_{1,2}(w^2), \\
        -w \, \mathbf{b}\S* \, \mathbf{d}\S
        &= -w \, \mathbf{b}\S* \left( \eye + w^2 \, \mathbf{A}\SF \, \mathbf{A}\FS + \ldots \right) \vecone_{\mathbf{s}\S}
        &&= -w - w^3 \, p_{2,1}(w^2),
    \end{alignat*}
    where $p_{1,2}$ and $p_{2,1}$ are polynomials.  Note the consistency assumption implies $\mathbf{b}\F* \vecone_{\mathbf{s}\F} = \mathbf{b}\S* \vecone_{\mathbf{s}\S} = 1$ and is used to determine the coefficient multiplying the $w$ terms.  Now the stability matrix can be written in the form
    \begin{equation*}
        R_2(w) = \begin{bmatrix}
            1 - w^2 \, p_{1,1}(w^2) & w + w^3 \, p_{1,2}(w^2) \\
            -w - w^3 \, p_{2,1}(w^2) & 1 - w^2 \, p_{2,2}(w^2)
        \end{bmatrix},
    \end{equation*}
    where $p_{1,1}$, and $p_{2,2}$ are also polynomials.
    
    Suppose by means of contradiction that the method is A-stable.  Consider the trace of the stability matrix:
    \begin{equation*}
        \tr(R_2(w)) = 2 - w^2 \, ( p_{1,1}(w^2) + p_{2,2}(w^2) ).
    \end{equation*}
    In order to avoid an eigenvalue of $R_2(w)$ being unbounded in $w$, we must have that $p_{2,2}(w^2) = -p_{1,1}(w^2)$.  Using this necessary condition, the determinant is
    \begin{align*}
        \det(R_2(w)) &= (1 - w^2 \, p_{1,1}(w^2)) (1 + w^2 \, p_{1,1}(w^2)) \\
        &\quad - (w + w^3 \, p_{1,2}(w^2)) (-w - w^3 \, p_{2,1}(w^2)) \\
        &= 1 + w^2 + \order{w^4}.
    \end{align*}
    Since the determinant grows unbounded in $w$, the spectral radius can be made arbitrarily large.  This is a contradiction.  Therefore, the method cannot be A-stable for the real 2D test problem.
    \qed
\end{proof}

\subsection{\cmrgark*{} Scalar Stability}

Directly using the general stability formula \cref{eq:scalar_stability} on an MrGARK method requires inverting a matrix of size $\mathbf{s} \times \mathbf{s}$.  When trying to analyze or visualize the linear stability for large $M$, this becomes very expensive.  Fortunately, the particular structure of \cmrgark{} MrGARK methods allows for an explicit derivation of the scalar stability function using only matrices of size $s \times s$.

For the base method $(A, b, c)$, let $\Rint(z)$ be the internal stability function:
\begin{equation*}
    \Rint(z) = \left( \eye{s} - z \, A \right)^{-1} \vecone_s.
\end{equation*}
We now seek to find the scalar internal stability of a \cmrgark{} MrGARK method.  Let $z = z\F + z\S$.  Then the first macro-step \cref{eq:MrGARK_CS:macro-step} is composed of traditional Runge--Kutta stages and is simply
\begin{equation} \label{eq:MrGARK_CS_stability_macro}
    Y = y_n \, \Rint(z)
\end{equation}
for the scalar linear test problem.  The $\lambda$-th fast micro-step \cref{eq:MrGARK_CS:micro-step} has stages defined by the recurrence relation
\begin{equation*}
    Y\FL = y_n \, \vecone_s + \frac{z\F}{M} \sum_{k=1}^{\lambda-1} \vecone_s \, b^T \, Y\FL[k] + \frac{z\F}{M} \, A \, Y\FL + z\S \, A\FSL \, Y.
\end{equation*}
Solving for $Y\FL$ explicitly is equivalent to solving the following linear system via block forward substitution:
\begin{equation*}
    \begin{bmatrix}
        I - \frac{z\F}{M} \, A & \ldots & 0 \\
        \vdots & \ddots & \vdots \\
        -\frac{z\F}{M} \, \vecone_s \, b^T & \ldots & I - \frac{z\F}{M} \, A
    \end{bmatrix}
    \begin{bmatrix}
        Y\FL[1] \\ \vdots \\ Y\FL[M]
    \end{bmatrix}
    = \begin{bmatrix}
        y_n \vecone_{s} + z\S \, A\FSL[1] \, Y \\
        \vdots \\
        y_n \vecone_{s} + z\S \, A\FSL[M] \, Y
    \end{bmatrix}.
\end{equation*}
This yields
\begin{equation} \label{eq:MrGARK_CS_stability_micro}
    \begin{split}
        & \quad Y\FL \\
        &= \frac{z\F}{M} \, \Rint \mleft( \frac{z\F}{M} \mright) \, b^T \left( I - \frac{z\F}{M} A \right)^{-1} \, \sum_{k=1}^{\lambda-1} R \mleft( \frac{z\F}{M} \mright)^{\lambda-1-k} \left( y_n \, \vecone_s \right. \\
        & \quad \left. + z\S \, A\FSL[k] \, Y \right) + \left( I - \frac{z\F}{M} \, A \right)^{-1} \left( y_n \, \vecone_s + z\S \, A\FSL \, Y \right) \\
        &= \left( \frac{z\F \, z\S}{M} \, \Rint \mleft( \frac{z\F}{M} \mright) \, b^T \left( I - \frac{z\F}{M} \, A \right)^{-1} \, \sum_{k=1}^{\lambda-1} R \mleft( \frac{z\F}{M} \mright)^{\lambda-1-k} \, A\FSL[k] \, \Rint(z) \right. \\
        & \quad \left. + z\S \left( I - \frac{z\F}{M} \, A \right)^{-1} \, A\FSL \, \Rint(z) + R \mleft( \frac{z\F}{M} \mright)^{\lambda-1} \, \Rint \mleft( \frac{z\F}{M} \mright) \right) y_n.
    \end{split}
\end{equation}
Together, \cref{eq:MrGARK_CS_stability_macro,eq:MrGARK_CS_stability_micro} form the internal stability for a \cmrgark{} MrGARK method.  With this in hand, the scalar linear stability function \cref{eq:scalar_stability} can be derived:
\begin{equation*}
    \begin{split}
        & \quad R_1 \mleft( z\F, z\S \mright) \\
        &= 1 + \frac{z\F}{M} \, \sum_{\lambda=1}^M \, b^T \, Y\FSL + z\S \, b^T \, Y \\
        &= R \mleft( \frac{z\F}{M} \mright)^M  + z\S b^T  \Rint \mleft( z \mright)\\
        & \quad + \frac{z\S \, z\F}{M} \, b^T \left( \eye{s} - \frac{z\F}{M} \, A \right)^{-1} \, \sum_{\lambda = 1}^M R \mleft( \frac{z\F}{M} \mright)^{M - \lambda} \, A\FSL \, \Rint \mleft( z \mright).
    \end{split}
\end{equation*}

If $A$ is invertible, then $\Rint(-\infty) = 0_s$ and
\begin{align*}
    \lim_{z\F \to -\infty} R_1 \mleft( z\F, z\S \mright) &= R(-\infty)^M  + z\S \left( b^T - b^T \, A^{-1} \,  A\FSL[M] \right) \Rint \mleft( -\infty \mright) \\
    &= R(-\infty)^M.
\end{align*}

The other limit is more difficult to approach directly, so we consider first the internal stability \cref{eq:MrGARK_CS_stability_micro}.  Starting with with the first micro-step, we have that
\begin{equation*}
    \lim_{z\S \to -\infty} Y\FL[1] = \left( I - \frac{z\F}{M} A \right)^{-1} \left( \vecone_s - A\FSL[1] A^{-1} \vecone_s \right) y_n.
\end{equation*}
This suggests the condition $A\FSL[1] \, A^{-1} \, \vecone_s = \vecone_s$ to ensure the stage values go to zero in the limit.  Now we can use an inductive argument to generalize this condition for the remaining micro-step stages.  Assume that $\lim_{z\S \to -\infty} Y\FL[\ell] = 0$ for $\ell = 1, \dots, \lambda - 1$.  Then
\begin{equation*}
    \lim_{z\S \to -\infty} Y\FL[\lambda] = \left( \eye{s} - \frac{z\F}{M} \, A \right)^{-1} \left( \vecone_s - A\FSL[\lambda] \, A^{-1} \, \vecone_s \right) \, y_n.
\end{equation*}
This suggests the condition
\begin{equation} \label{eq:CS_zs_limit_condition}
    A\FSL \, A^{-1} \, \vecone_s = \vecone_s, \qquad \lambda = 1, \dots, M
\end{equation}
to ensure all stages go to zero in the limit.  Further \cref{eq:CS_zs_limit_condition} leads to the result
\begin{equation*}
    \lim_{z\S \to -\infty} R_1 \mleft( z\F, z\S \mright) = R(-\infty) + \frac{z\F}{M} \, \sum_{\lambda=1}^M \, b^T \, Y\FSL = R(-\infty).
\end{equation*}

\section{Numerical Solution of Implicit Stage Equations}
\label{sec:newton}

The key to an efficient implicit GARK method is an efficient Newton iteration.  Written compactly, the stage equations are
\begin{equation} \label{eq:GARK_compact}
    \widehat{Y} = \vecone_{\mathbf{s}} \otimes y_n + H \, (\mathbf{A} \otimes \eye{d}) \, \widehat{f}\left( \widehat{Y} \right),
\end{equation}
where
\begin{equation}
    \widehat{Y} = \begin{bmatrix}
        Y\F \\ Y\S
    \end{bmatrix}, \qquad
    \widehat{f}\left( \widehat{Y} \right) = \begin{bmatrix}
        f\F(Y\F) \\ f\S(Y\S)
    \end{bmatrix}.
\end{equation}
Applying Newton's method to solve for the stages yields the iterative procedure
\begin{subequations} \label{eq:newton}
    \begin{align}
        \left( \eye{\mathbf{s}} \otimes \eye{d} - H \, (\mathbf{A} \otimes \eye{d}) \, \widehat{J} \right) \, \delta
        &= -\widehat{Y} + \vecone_{\mathbf{s}} \otimes y_n + H \, (\mathbf{A} \otimes \eye{d}) \widehat{f}\left( \widehat{Y} \right), \label{eq:newton:solve} \\
        \widehat{Y} &= \widehat{Y} + \delta, \label{eq:newton:update}
    \end{align}
\end{subequations}
with
\begin{equation}
    \widehat{J} = \diag{\left( J\F_1, \ldots, J\F_{s\F}, J\S_1, \ldots, J\S_{s\S} \right)},
\end{equation}
and $J\comp{\sigma}_i = \pdv{f\comp{\sigma}}{y} \mleft( Y\comp{\sigma}_i \mright)$ for $\sigma \in \{\slow, \fast\}$.

In single rate Newton iterations, it is common to evaluate the Jacobian once at $y_n$ and use it across all stages which yields a cheaper modified Newton's method.  A similar strategy can be employed for each partition's Jacobian in a GARK Newton iteration.  For multirate methods, it might be beneficial to reevaluate the fast Jacobian at each micro-step and keep the slow Jacobian across the entire macro-step.

We note that \cref{eq:newton} serves mostly theoretical purposes, as it is impractically expensive and rarely necessary to simultaneously solve for all $\mathbf{s}$ stages.  All methods presented in \cref{sec:methods}, for example, require solving nonlinear systems with dimension no larger than $d$.  In this section, we will explore techniques and method structures that allow for these efficient implementations of Newton iterations.  In the cost analyses we present, matrix decompositions involving the Jacobians are assumed to be the dominant cost of a step.


\subsection{Decoupled Methods}
\label{sec:newton:decoupled}

As described in \cref{sec:MrGARK:standard}, decoupled methods only have implicitness in the base methods.  For this subsection, we will assume both base methods are diagonally implicit which seems to be the most practical structure for decoupled implicit methods.  Now, each of the $\mathbf{s}$ method stages defines a $d$-dimensional nonlinear equation which can be solved sequentially for a cost of $\order{\mathbf{s} \, d^3}$, assuming direct methods are used.  If we further assume the slow matrix decomposition is reused across a multirate macro-step and the fast matrix decomposition is reused across a micro-step, the cost is reduced to $\order{M \, d^3}$.  It is important to note that the slow and fast Jacobians are likely to have simpler structures than the full Jacobian, and these structures can be exploited in the linear solves.

For the special case of component partitioned systems \cref{eq:ode_comp}, the linear solves are of the reduced dimensions $d\F$ and $d\S$.  In the most extreme case where each variable of a system forms a partition, a step would involve scalar Newton iterations for all variables and only the diagonal of the Jacobian of $f$ would be required.  We note, however, that this leads to an explosion in the number of coupling error terms and degraded stability.

\subsection{\cmrgark*{} Methods}
\label{sec:newton:CS}

\cmrgark*{} methods start by taking a full macro-step like a single rate Runge--Kutta method.  Consequently, the nonlinear equations for the stages can be solved just as they would for a single rate method.  When using Newton's method, the full, unpartitioned Jacobian is used.  It may be appropriate to loosen the solver tolerances of the fast variables for the compound step as they will be recomputed later \cite{verhoeven2006error}.  Although the remaining micro-steps are also implicitly defined, only $J\F_i$ is now involved in Newton iterations.  Assuming a diagonally implicit structure for the base method, these Newton iterations are of the form
\begin{align*}
    \left( \eye{d} - h \, a\FF_{i,i} \, J\F_i \right) \delta &= -Y\FL_i + \widetilde{y}_{n + (\lambda - 1)/M} + h \sum_{j = 1}^{s\F} a\FF_{i,j} \, f\F \mleft( Y_j\FL \mright) \\
    & \quad + H \, \sum_{j = 1}^{s\S} a\FSL_{i,j} \, f\S \mleft( Y_j \mright).
\end{align*}
We note that an accurate stage value predictor to start the Newton iterations can come from dense output of the compound step.

In an implementation where a decomposition of the full matrix is formed once and a decomposition for the fast matrix is formed at each micro-step, the total cost for one step is $\order{M \, d^3}$.  For component partitioned systems, this reduces to $\order{d^3 + M \, d\F[\times 3]}$.

\subsection{Stage Reducibility}
\label{sec:newton:stage-reduce}

Consider the simple methods defined by the GARK tableaus \cref{eq:tableau} below:
\begin{equation*}
    \begin{butchertableau}{c|c}
        1 & 1 \\ \hline
        1 & 1 \\ \hline
        1 & 1
    \end{butchertableau}
    \quad \text{and} \quad
    \begin{butchertableau}{c|c}
        \frac{1}{2} & 1 \\ \hline
        \frac{1}{2} & 1 \\ \hline
        1 & 1
    \end{butchertableau}.
\end{equation*}
The former is backward Euler cast into the GARK framework.  A direct application of \cref{eq:newton} would require solving linear systems of size $2 d$ when clearly solves of size $d$ can suffice.  Here, $Y\F_1 = Y\S_1$, and these stages fall back onto the traditional backward Euler stage $Y_1 = y_n + H f(Y_1)$.  The latter method, which is an additive Runge--Kutta (ARK) method cast into the GARK framework, also has $Y\F_1 = Y\S_1$.  \Cref{eq:newton:solve} can be simplified to
\begin{equation} \label{eq:ark_newton}
    \left( \eye{d} - \frac{H}{2} \, J\F_1 + H \, J\S_1 \right) \delta = -Y_1 + y_n + \frac{H}{2} \, f\F(Y_1) + H \, f\S(Y_1).
\end{equation}
More generally when a row of GARK coefficients is repeated in multiple partitions, the number of unknowns in \cref{eq:GARK_compact} and the dimension of the Newton iteration is reduced.  We call this \textit{stage reducibility}.  \cmrgark*{} methods, for example, have this property in the first $s$ stages.

In \Cref{sec:methods}, we develop new multirate coupling strategies that utilize this simplification.  An interesting property is that the solves involve matrices of the form $\eye{d} - h \, \gamma J\F_i - H \, \gamma J\S_i$.  Note $J\F_i$ is scaled by the micro-step, while $J\S_i$ is scaled by the macro-step.  If the multirate ratio is based on partition stiffness, then the scaled matrices should have similar spectral radii.  By damping the fast, stiff modes, the conditioning of this system can be much better than the traditional $\eye{d} - H \, \gamma J_i$.

\subsection{Low Rank Structure of Matrices in Newton Iteration}

When a GARK method has stage reducibility, $\mathbf{A}$ cannot be full rank due to at least one repeated row.  An alternative simplification arises by applying the Woodbury matrix identity to reduce the dimension of the linear solve.  Using the GARK method below, we demonstrate that this idea can be extended to a broader set of schemes:
\begin{equation*}
    \begin{butchertableau}{c|c}
        \frac{1}{2} & \frac{1}{2} \\ \hline
        1 & 1 \\ \hline
        1 & 1
    \end{butchertableau}.
\end{equation*}
We have the following simplification in the Newton iteration:
\begin{align*}
    & \quad \left( \eye{\mathbf{s}} \otimes \eye{d} - H \, (\mathbf{A} \otimes \eye{d}) \, \widehat{J} \right)^{-1} \\
    &= \left( \begin{bmatrix} \eye{d} & 0 \\ 0 & \eye{d} \end{bmatrix} - H \begin{bmatrix} \frac{1}{2} \, J\F_1 & \frac{1}{2} \, J\S_1 \\ J\F_1 & J\S_1 \end{bmatrix} \right)^{-1} \\
    &= \begin{bmatrix} \eye{d} & 0 \\ 0 & \eye{d} \end{bmatrix} + \begin{bmatrix} \frac{H}{2} \, \eye{d} \\ H \, \eye{d} \end{bmatrix} \left( \eye{d} - \frac{H}{2} \, J\F_1 - H \, J\S_1 \right)^{-1} \begin{bmatrix} J\F_1 & J\S_1 \end{bmatrix}.
\end{align*}
Compared to \cref{eq:ark_newton}, additional matrix-vector products are required, but ultimately, the same matrix inverse appears.  Thus, the potential to have improved conditioning is still present.

\section{Practical implicit MrGARK methods}
\label{sec:methods}

In this section, we present new implicit MrGARK methods of orders one to four.  All methods are telescopic and based on single singly diagonally implicit Runge--Kutta (SDIRK) methods.  At high order, coupling coefficients can become complicated rational functions of $\lambda$ and $M$.  In addition to listing the coefficients in this paper, a Mathematica notebook with the coefficients is provided in the supplementary materials to aid those implementing the methods.

\subsection{First order}

Multirate methods of order one have no coupling conditions which allows a great amount of freedom in deriving coefficients, but for implicit methods, stability does impose some important constraints.  \Cref{thm:first_order} eliminates one subset of first order methods from being scalar A-stable.

\begin{theorem} \label{thm:first_order}
  An internally consistent MrGARK method of order exactly one is only scalar A-stable for a finite number of multirate ratios.
\end{theorem}


\begin{proof}
    Using the internal consistency assumptions, the magnitude of the scalar stability function can be expanded as
    \begin{equation} \label{eq:stability_magnitude_expansion}
        \begin{split}
            \abs{R_1(i \, \omega\F \, y, i \, \omega\S \, y)}^2 
            &= 1 + y^2  \left( \omega\F + \omega\S \right)^2  \\
            & \quad - 2 y^2 \left( \left( \omega\F + \omega\S \right) \left( \omega\F \, \mathbf{b}\F* \, \mathbf{c}\F + \omega\S \, \mathbf{b}\S* \, \mathbf{c}\S \right) \right)\\
            & \quad + \order{y^4} \\
            &= 1 + y^2 \, p \mleft( \omega\F, \omega\S \mright) + \order{y^4}.
        \end{split}
    \end{equation}

    Let $H$ be the Hessian matrix of the homogeneous polynomial of degree two $p$.  Note that $\det(H) = -4 \left( r\F - r\S \right)^2$, where $r\F = \mathbf{b}\F* \, \mathbf{c}\F - \frac{1}{2}$ and $r\S = \mathbf{b}\S* \, \mathbf{c}\S - \frac{1}{2}$ which are the second order residuals.  These residuals cannot both be zero because the GARK method would be order two by internal consistency.  When one base method is order one and the other is higher order, these residuals must differ.  Otherwise, when both base methods have order one, $r\F$ is a function of $M$ which approaches zero while $r\S$ is a fixed nonzero constant.  For all but a finite set of $M$, these residuals must differ.  Whenever the residuals differ, $p$ is saddle-shaped, and there exist $\omega\F$ and $\omega\S$ such that the polynomial is positive.  For sufficiently small values of $y$, the positive $y^2 \, p \mleft( \omega\F, \omega\S \mright)$ term will dominate the $\order{y^4}$ term in \cref{eq:stability_magnitude_expansion}.  Thus, for all but a finite set of $M$, there are $\omega\F$, $\omega\S$, and $y$ such that $\abs{R_1(i \, \omega\F \, y, i \, \omega\S \, y)} > 1$.
    \qed
\end{proof}

\begin{remark}
    Note that \cref{thm:first_order} imposes no restriction on the multirate strategy. It only requires the defining characteristic of a multirate method: the fast error asymptotically approaches zero as $M$ increases.
\end{remark}

At first order, the natural choice for an implicit base method is backward Euler.  There is currently a plethora of multirate backward Euler schemes in the literature (see \cite{sand1992stability,verhoeven2006general,zhao2016asynchronous,hachtel2019multirate}).  These schemes feature nearly all the different combinations of coupled or decoupled, internal consistency or internal inconsistency, and parallel or sequential methods.  In the search for a multirate backward Euler method with excellent stability and accuracy properties, we developed the coupling strategy given by the following standard MrGARK coupling coefficients:
\begin{equation} \label{eq:be_coupled}
    A\FSL = \begin{bmatrix} \begin{cases}
        0 & \lambda < \frac{M}{2} \\
        1 & \text{otherwise}
    \end{cases} \end{bmatrix},
    \qquad
    A\SFL = \begin{bmatrix} \begin{cases}
        1 & \lambda \leq \frac{M + 1}{2} \\
        0 & \text{otherwise}
    \end{cases} \end{bmatrix}.
\end{equation}
This method has one coupled stage, but with stage reducibility (\cref{sec:newton:stage-reduce}), and all other stages are decoupled.  Further, it is internally inconsistent and is scalar L- and algebraically stable for all $M$.  A decoupled counterpart is given by the following coupling coefficients:
\begin{equation} \label{eq:be_decoupled}
    A\FSL = \begin{bmatrix} \begin{cases}
        0 & \lambda \leq \frac{M}{2} \\
        1 & \text{otherwise}
    \end{cases} \end{bmatrix},
    \qquad
    A\SFL = \begin{bmatrix} \begin{cases}
        1 & \lambda \leq \frac{M}{2} \\
        0 & \text{otherwise}
    \end{cases} \end{bmatrix}.
\end{equation}
This method is internally inconsistent, has no second order coupling error when $M$ is even, and is scalar L- and algebraically stable for all $M$.

We note this method is closely connected to the following subcycled Strang splitting \cite{strang1968construction}:
\begin{equation*}
    \varphi_H^f = \left( \varphi_h^{f\F} \right)^{M/2} \circ \varphi_H^{f\S} \circ \left( \varphi_h^{f\F} \right)^{M/2} + \order{H^2}.
\end{equation*}
Here, the operator $\varphi_t^g$ maps an initial condition for the ODE $y' = g(y)$ to the solution at time $t$.  If we approximate these exact ODE solutions with one step of the backward Euler method, we recover the decoupled multirate backward Euler scheme \cref{eq:be_decoupled}. 

\subsection{Second order}

The simplest second order base method is the one stage implicit midpoint method:
\begin{equation*}
    \begin{butchertableau}{c|c}
        \frac{1}{2} & \frac{1}{2} \\ \hline
        & 1
    \end{butchertableau}.
\end{equation*}
The standard MrGARK coupling coefficients
\begin{equation*}
    A\FSL = \begin{bmatrix}
        0 & \lambda < L \\
        \frac{1}{2} & \lambda = L \\
        1 & \lambda > L
    \end{bmatrix}, 
    \qquad
    A\SFL = \begin{bmatrix}
        1 & \lambda < L \\
        \frac{1}{2} & \lambda = L \\
        0 & \lambda > L
    \end{bmatrix},
\end{equation*}
for odd $M$ and $L = \frac{M + 1}{2}$  give a coupled multirate midpoint method.  Similar to the coupled backward Euler method \cref{eq:be_coupled}, one stage is coupled but with stage reducibility, and all other stages are decoupled.  Reusing the coupling coefficients \cref{eq:be_decoupled} with even $M$ and the midpoint method as the base, we derive a decoupled multirate midpoint method.  Notably, both schemes maintain the algebraic stability, symmetry, and symplecticity \cite{zanna2020discrete} of the midpoint method.  With only odd order terms appearing in the error expansion, they can be used to build efficient multirate extrapolation methods.

We also consider the L-stable, order two SDIRK base method from \cite{alexander1977diagonally}
\begin{equation} \label{eq:sdirk2}
    \begin{butchertableau}{c|cc}
        \gamma & \gamma & 0 \\
        1 & 1 - \gamma & \gamma \\ \hline
        & 1 - \gamma & \gamma \\ \hline
        & \frac{3}{5} & \frac{2}{5}
    \end{butchertableau},
\end{equation}
with $\gamma = 1 - 1 / \sqrt{2}$.  For this base method, an internally consistent standard MrGARK method must have at least one coupled stage.  Enforcing stiff accuracy for both partitions uniquely determines a lightly coupled method:
\begin{equation} \label{eq:sdirk2_stiffly_accurate}
    \begin{split}
        A\FSL &= \begin{bmatrix}
            \frac{\lambda - 1 + \gamma}{M} & 0 \\
            \begin{cases}
                1 - \gamma & \lambda = M \\
                \frac{\lambda}{M} & \text{otherwise}
            \end{cases} &
            \begin{cases}
                \gamma & \lambda = M \\
                0 & \text{otherwise}
            \end{cases} \\
        \end{bmatrix}, \\
        A\SFL &= \begin{bmatrix}
            \begin{cases}
                M \gamma & \lambda = 1 \\
                0 & \text{otherwise}
            \end{cases} & 0 \\
            1 - \gamma & \gamma
        \end{bmatrix}.
    \end{split}
\end{equation}
For this method, the first slow and fast stages are coupled, but with low rank structure.  The last slow and fast stages are also coupled, but with stage reducibility.  All other stages are decoupled.  Another coupling strategy is that of Kv\ae rn\o{} and Rentrop \cite{kvaerno1999low,gunther2016multirate} in which the first micro-step and the macro-step are computed together.  The following coupling coefficients take this approach and also enforce \cref{eq:scalar_rinf}:
\begin{equation} \label{eq:sdrik2_kr}
    \begin{split}
        A\FSL &= \begin{bmatrix}
            \frac{\gamma (2 \lambda - 1)}{M}
            & \begin{cases}
                0 & \lambda = 1 \\
                \frac{(\lambda - 1) (1 - 2 \gamma)}{M} & \lambda > 1
            \end{cases} \\
            \frac{1 - 3 \gamma + 2 \gamma \lambda}{M} & \frac{3 \gamma -1 + (1 - 2 \gamma) \lambda}{M}
        \end{bmatrix}, \\
        A\SFL &= \begin{cases}
            M \begin{bmatrix}
                \gamma & 0 \\ 1 - \gamma & \gamma
            \end{bmatrix} & \lambda = 1 \\
            0_{2 \times 2} & \text{otherwise}
        \end{cases}.
    \end{split}
\end{equation}
Here, the first two fast and slow stages are coupled, but with low rank structure.

An interesting feature of these two coupled methods is their scalar linear stability functions coincide for all $M$.  Unfortunately instabilities appear near the origin as $M$ increases.  At $M=2$, the methods are only scalar $L(\ang{69.2})$-stable (as defined in \cref{def:scalar_stability_alpha}), and by $M = 6$ they are not even scalar $L(\ang{0})$-stable.  While \cref{eq:sdirk2_stiffly_accurate,eq:sdrik2_kr} may be effective for some problems, we cannot recommend them as general-purpose multirate methods.  This is a surprising result as these seemingly reasonable coupling structures lead to methods with worse stability and a more expensive implementation than the decoupled multirate midpoint method.  As is the case with order one, it appears that internal consistency negatively affects the stability. 

The following \cmrgark{} method, which we will call \cmrgark{} MrGARK SDIRK2, can be derived from stability condition \cref{eq:CS_zs_limit_condition} and internal consistency:
\begin{equation} \label{eq:mr_sdirk2}
    A\FSL = \begin{bmatrix}
         \frac{-\gamma  ((M-2) \gamma +3)+(2 \gamma -1) \lambda +1}{M (\gamma -1)} & \frac{\gamma  ((M-1) \gamma -\lambda +1)}{M (\gamma -1)} \\
         \frac{M \gamma ^2-2 \lambda  \gamma +\lambda }{M-M \gamma } & \frac{\gamma  (M \gamma -\lambda )}{M (\gamma -1)}
    \end{bmatrix}.
\end{equation}
The angles of L$(\alpha)$-stability for several values of $M$ are listed in \cref{tab:method_stability}.  Unlike the aforementioned internally consistent methods of order two, \cmrgark{} MrGARK SDIRK2 maintains a wide angle when $M$ is large.

\subsection{Higher order methods}

\todo{Method coefficients have been moved to the appendices}
Following the results at order two, we focus our search for implicit multirate GARK methods of orders three and four on \cmrgark{} methods.  Stiff accuracy, L-stability of the base method, and \cref{eq:CS_zs_limit_condition} are enforced to ensure acceptable stability properties.  Internal consistency drastically reduces the number of order conditions at these higher orders and allows us to use the simplified conditions in \cref{eq:cs_order_conditions}.  Finally, coupling coefficients are derived such as to be bounded functions of $\lambda$ and $M$.  Without this, methods are susceptible to catastrophic cancellation when $M$ is large.  In \cref{app:mr_sdirk3,app:mr_sdirk4}, we have listed the coefficients of the third and fourth order methods we derived with the aforementioned constraints.

\ifreport
Despite the stability issues observed at second order, we also consider a third order implicit multirate method with Kv\ae rn\o --Rentrop coupling using the following algebraically stable base method from \cite{norsett1974semi}:
\begin{equation} \label{eq:dirk2}
    \begin{butchertableau}{c|cccc}
        \gamma & \gamma \\
        1 - \gamma  & 1 - 2\gamma & \gamma \\
        \hline
        & 1/2 &  1/2 \\
    \end{butchertableau}, \qquad \gamma = \frac{3 + \sqrt{3}}{6}.
\end{equation}
Following the approach in \cite{kvaerno1999low,gunther2016multirate}, the slow to fast coupling is chosen to be 
\begin{subequations}
    \begin{align*}
        A\FSL[\lambda + 1] &= \frac{1}{M} \left( A\FS + F(\lambda) \right),\\
        F(\lambda) &= \vecone_{s\F} \begin{bmatrix} \eta_1(\lambda) & \dots & \eta_{s\S}(\lambda) \end{bmatrix}, \qquad \lambda = 0,\ldots,M-1,
    \end{align*}
\end{subequations}
where the $\eta_j$ satisfy $\sum_{j=1}^{s\S} \eta_j(\lambda) = \lambda$.
This results in the internal consistency condition reducing to
\begin{equation} \label{eq:kv1}
    A\FS \, \vecone_{s\S} = c,
\end{equation}
and the third order coupling condition becoming
\begin{equation} \label{eq:kv2}
    \frac{M}{6} = b^T \left( A\FS + \frac{1}{M} \sum_{\lambda=1}^{M} F(\lambda) \right) c.
\end{equation}
%
At third order, this approach creates coefficients that grow unbounded with $M$.  Moreover, we were unable to find an A$(\ang{0})$-stable method satisfying the constraints \cref{eq:kv1,eq:kv2}.
\fi

\subsection{Scalar stability of new \cmrgark{} methods}

The scalar stability of \cmrgark{} methods \cref{eq:mr_sdirk2,eq:mr_sdirk3,eq:mr_sdirk4} are summarized in \cref{tab:method_stability}.  In all cases, the methods are just a few degrees short of scalar L-stability.  As $M$ increases, the stability angles decrease by less than $\ang{2}$ before stabilizing.

\begin{table}[ht!]
    \centering
    \begin{tabular}{c|cccccc}
        \shortstack{\cmrgark* \\ method} & $M = 2$ & $M = 3$ & $M = 4$ & $M = 8$ & $M = 16$ & $M = 32$ \\ \hline
        SDIRK2 from \labelcref{eq:mr_sdirk2} & \ang{84.6} & \ang{83.5} & \ang{83.2} & \ang{83.0} & \ang{83.0} & \ang{83.0} \\
        SDIRK3 from \labelcref{eq:mr_sdirk3} & \ang{88.6} & \ang{87.8} & \ang{87.3} & \ang{86.9} & \ang{86.8} & \ang{86.8} \\
        SDIRK4 from \labelcref{eq:mr_sdirk4} & \ang{81.7} & \ang{81.2} & \ang{81.2} & \ang{81.2} & \ang{81.2} & \ang{81.2} \\
    \end{tabular}
    \caption{Scalar $L(\alpha)$-stability (as defined in \cref{def:scalar_stability_alpha}) for new \cmrgark{} MrGARK methods.}
    \label{tab:method_stability}
\end{table}
\section{Numerical Experiments}
\label{sec:experiments}

In this section, we use the new methods to integrate two test problems.  First, the CUSP model is used to verify the order of accuracy.  Next, the inverter chain model is used to compare the performance of multirate methods against single rate and implicit-explicit (IMEX) methods.

\subsection{CUSP Model}
\label{subsec:cusp}

The CUSP model, as reported in \cite[Chapter IV.10]{Hairer_book_II}, is a reaction-diffusion model defined with the equations
\begin{align} \label{eq:cusp}
	\begin{split}
	 	\pdv{y}{t} &= -\frac{1}{\varepsilon} \left( y^3 + a \, y + b \right) + \sigma \, \pdv[2]{y}{x}, \\
	 	\pdv{a}{t} &=  b+ 0.07 \, v +  \sigma \, \pdv[2]{a}{x},\\
	 	\pdv{b}{t} &= b \, (1-a^2) - a- 0.4 \, y + 0.035 \, v + \sigma \, \pdv[2]{b}{x},
	\end{split}
\end{align}
where $v = \frac{u}{u+ 0.1}$ and $u = (y-0.7) \, (y-1.3)$.  The parameters are $\sigma = \frac{1}{144}$ and $\varepsilon = 10^{-4}$, which makes the problem stiff.  \Cref{eq:cusp} is integrated from $t=0$ to $t=1.1$ over the spatial domain $x \in [0, 1]$.  In our numerical experiments, we use second order central finite differences on a uniform mesh with $N=32$ points and periodic boundary conditions. The initial conditions are
\begin{equation*}
	y_i(0) =0, \quad a_i(0) = -2 \cos(\frac{2\pi i}{N}), \quad b_i(0) = 2\sin(\frac{2\pi i}{N}), \quad \text{for} \quad  i=1,\dots,N.
\end{equation*}

The splitting of the right-hand side function is done over the physics: diffusion is considered as the slow function, and the remaining reactive terms are the fast function. The MATLAB implementation of the CUSP problem is available in \cite{Sandu_2019_ODE-tests}.  We use MATLAB's \texttt{ode15s} to compute a high-accuracy reference solution with absolute and relative tolerances set to $10^{-13}$.  Error is measured as the 2-norm of the difference of the numerical solution and this reference solution at time $t = 1.1$.


\Cref{fig:convergence} shows convergence results for the decoupled multirate midpoint method and \cmrgark{} MrGARK SDIRK methods of orders two, three, and four using a range of multirate ratios.  For \cmrgark{} MrGARK SDIRK4, the numerical rate of convergence is slightly higher than the nominal order.  In all other cases, the numerical orders closely match the theoretical ones.  We also note that for a fixed number of steps, the error decreases as the multirate ratio increases as expected.

\begin{figure}[ht!]
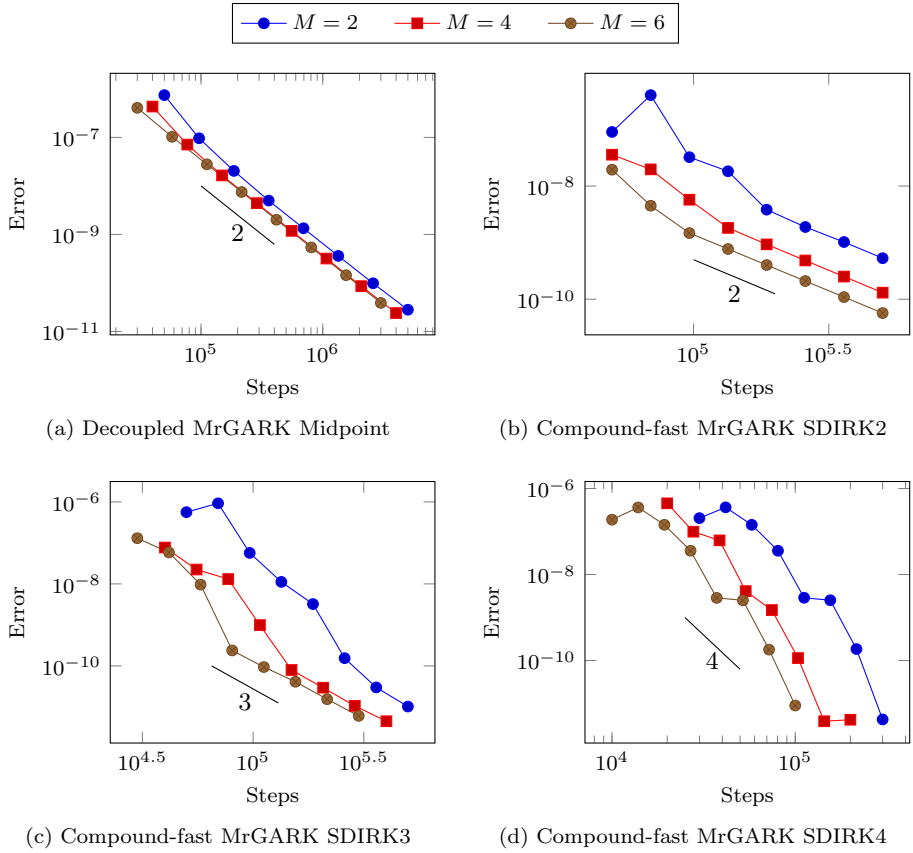

	\centering
 	\ref*{conv}
 	
 	\subfloat[Decoupled MrGARK Midpoint]{\convplot{0.48\textwidth}{midpoint.txt}{1e5,1e-8}{4e5,6.25e-10}{2}*}
    \hfill
  	\subfloat[\cmrgark*{} MrGARK SDIRK2]{\convplot{0.48\textwidth}{cf2.txt}{1e5,5.0e-10}{2e5,0.25*5e-10}{2}}
    
	\subfloat[\cmrgark*{} MrGARK SDIRK3]{\convplot{0.48\textwidth}{cf3.txt}{65e3,1.0e-10}{130e3,0.125*1.e-10}{3}}
    \hfill
    \subfloat[\cmrgark*{} MrGARK SDIRK4]{\convplot{0.48\textwidth}{cf4.txt}{2.5e4,1.0e-9}{5e4,(1/16*1.0e-9}{4}}
    
	\caption{Error vs.  number of macro-steps for the decoupled midpoint method \cref{eq:be_decoupled} and \cmrgark{} methods \cref{eq:mr_sdirk2,eq:mr_sdirk3,eq:mr_sdirk4} applied to the CUSP problem \cref{eq:cusp}.  Reference slopes are included to compare with the numerical orders.}
	\label{fig:convergence}
\end{figure}



\subsection{Inverter Chain Model}

We also consider the inverter chain model of \cite{kvaerno1999low,Bartel2002} given by the equations
\begin{equation} \label{eq:inverter_chain}
    \begin{split}
        U'_1 &= U_{op} - U_1 - \Gamma \, g(U_{in}, U_1, U_0), \\
        U'_i &= U_{op} - U_i - \Gamma \, g(U_{i-1}, U_i, U_0), \qquad i = 2, \dots m,
    \end{split}
\end{equation}
with
\begin{align*}
    g(U_g, U_D, U_S) = (\max(U_G - U_S - U_T, 0))^2 - (\max(U_G - U_D - U_T, 0))^2.
\end{align*}
It models the propagation of the input signal
\begin{align*}
    U_{in}(t) = \begin{cases}
    t - 5 & 5 \le t \le 10 \\
    5 & 10 \le t \le 15 \\
    \frac{5}{2} (17 - t) & 15 \le t \le 17 \\
    0 & \text{otherwise}
    \end{cases}
\end{align*}
through a sequence of $m$ metal-oxide-semiconductor field-effect transistor (MOSFET) inverters.  The ground voltage is $U_0 = 0$, the operating voltage is $U_{op} = 5$, and the threshold voltage separating the on and off states is $U_T = 1$.  Stiffness is controlled by $\Gamma$ and is taken to be $100$ as in the stiff case used in \cite{Bartel2002}.  The initial conditions of the system are
\begin{align*}
    U_i(0) = \begin{cases}
        6.246 \times 10^{-3} & i \text{ even} \\
        5 & i \text{ odd}
    \end{cases}.
\end{align*}
For the numerical experiments, we use $m = 500$ inverters and a timespan of $[0, 120]$ to allow the signal to reach the end of the chain.

In this numerical experiment, we compare the performance of three types of methods on the inverter chain: single rate, IMEX Runge--Kutta, and \cmrgark{} MrGARK.  While \cmrgark{} MrGARK can use dynamic partitioning to select the fast inverters as described in \cref{sec:MrGARK:cs}, there is not a direct analog for IMEX Runge--Kutta methods.  For this reason, we use a fixed, time-dependent partitioning that follows the propagation of the signal though the chain.  Inverters with indices in the range
\begin{equation*}
    \left[\min(\max(1, \floor{4.75 t - 95}), m + 1),
    \min(\max(0, \floor{4.75 t - 15}), m) \right]
\end{equation*}
form the fast partition.  Only these inverters are treated with a microstep by multirate schemes and implicitly by IMEX schemes.

\todo{This paragraph has been rewritten to be more clear}
At second order, we consider the performance of \cmrgark{} MrGARK SDIRK2 from \cref{eq:mr_sdirk2}.  The primary baseline is its base method: SDIRK2 from \cref{eq:sdirk2}.  SDIRK2 is a traditional Runge--Kutta method and treats all inverters with the same timestep.  We also test the IMEX Runge--Kutta scheme ARS(2,3,2) from \cite[Section 2.5]{ascher1997implicit}.  Note that the implicit part of ARS(2,3,2) is SDIRK2, which makes for a fair comparison among all three second order methods.  At third order, we use \cmrgark{} MrGARK SDIRK3 from \cref{eq:mr_sdirk3}, its base method SDIRK3 from \cref{eq:sdirk3}, and ARS(3,4,3) from \cite[Section 2.7]{ascher1997implicit}.  The implicit part of ARS(3,4,3) is SDIRK3.  At fourth order, we use \cmrgark{} MrGARK SDIRK4 from \cref{eq:mr_sdirk4}.  SDIRK4 from \cite[pg. 100]{Hairer_book_II} is slightly more optimized than \cref{eq:sr_sdirk4}, so we use that for the traditional Runge--Kutta baseline.  Finally, ARK4(3)6L[2]SA from \cite{kennedy2001additive} is used as the fourth order IMEX scheme.  The multirate ratios we use are $M = 14, 10, 6$ for orders two, three, and four, respectively.

A serial C implementation of the inverter chain and integrators was run on the Cascades cluster managed by Advanced Research Computing (ARC) at Virginia Tech.  In the experiment, the error and runtime were recorded for a range of eight stepsizes for all nine methods.  Error is computed in the infinity-norm with respect to a high-accuracy reference solution.  \Cref{fig:performance} plots the timing results.  At orders two and three, we can see the \cmrgark{} MrGARK methods reach a fixed accuracy four to six times faster than the single rate methods and are slightly more efficient than the IMEX methods.  The fourth order multirate and IMEX methods have similar performance and are approximately three times faster than the single rate method.

\begin{figure}[ht!]
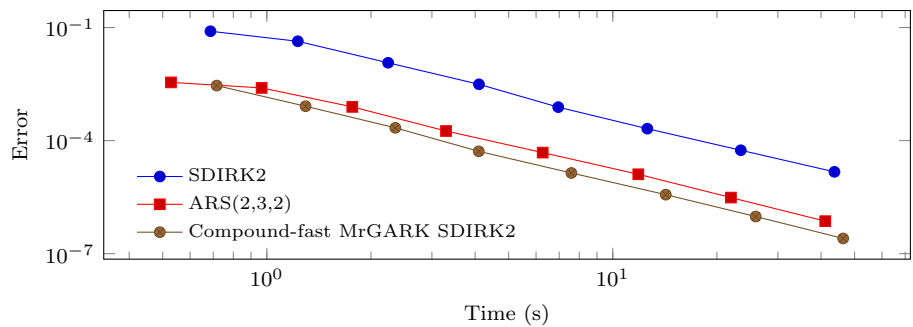
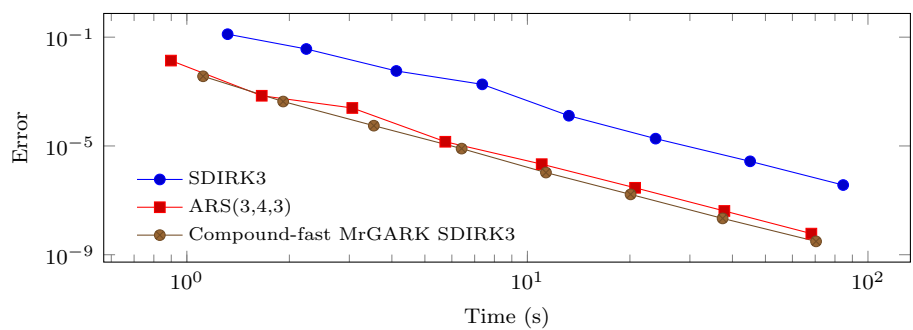
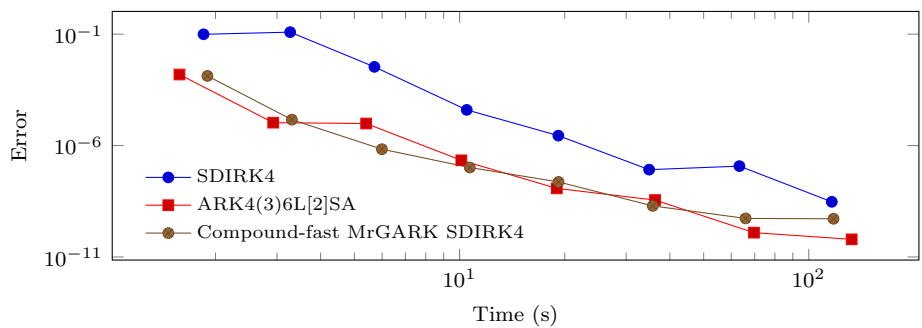

	\centering
	
	\subfloat[Second order]{\wpplot{SDIRK2, {ARS(2,3,2)}, \cmrgark*{} MrGARK SDIRK2}{cf2.txt}}
    \\
    \subfloat[Third order]{\wpplot{SDIRK3, {ARS(3,4,3)}, \cmrgark*{} MrGARK SDIRK3}{cf3.txt}}
    \\
    \subfloat[Fourth order]{\wpplot{SDIRK4, {ARK4(3)6L[2]SA}, \cmrgark*{} MrGARK SDIRK4}{cf4.txt}}

	\caption{Error vs. time for single rate, IMEX, and \cmrgark{} methods applied to the inverter chain problem \cref{eq:inverter_chain}}
	\label{fig:performance}
\end{figure}


Despite the similar performance of the IMEX and multirate methods, the number of steps required to reach a desired accuracy is very different.  The stiffness of the inverter chain problem, even in the slow partition, forces the IMEX schemes to take relatively small timesteps.  \Cref{tab:stepsizes} lists the largest timestep each of the tested methods could take.  For comparison, we have also included single rate explicit methods.  In particular, we use Ralston's optimal second and third order methods \cite{ralston1962runge} and the classical fourth order Runge--Kutta method.  Note that the IMEX methods have an explicit-like stepsize restriction for this problem.  The \cmrgark{} MrGARK methods have the same maximum stepsize as the implicit single rate methods, which indicates stiffness in the slow partition is likely limiting the stepsize.

\begin{table}[ht!]
    \centering
    \begin{tabular}{c|cccc}
        & \shortstack{Single rate\\implicit} & \shortstack{Single rate\\explicit} & IMEX & \cmrgark*{} \\ \hline
        Order 2 & $7.1 \times 10^{-2}$ & $3.1 \times 10^{-3}$ & $3.2 \times 10^{-3}$ & $7.1 \times 10^{-2}$ \\
        Order 3 & $5.2 \times 10^{-2}$ & $3.2 \times 10^{-3}$ & $3.5 \times 10^{-3}$ & $5.2 \times 10^{-2}$ \\
        Order 4 & $6.0 \times 10^{-2}$ & $3.4 \times 10^{-3}$ & $5.2 \times 10^{-3}$ & $6.0 \times 10^{-2}$
    \end{tabular}
    \caption{Approximate largest stepsizes to ensure stability and convergence of Newton iterations for the inverter chain problem \cref{eq:inverter_chain}.}
    \label{tab:stepsizes}
\end{table}
\section{Conclusions}
\label{sec:conclusion}

In this work, we have explored multirate Runge--Kutta methods in which all time-scales are treated implicitly.  By taking different timesteps for different partitions of an ODE, these methods can more efficiently integrate stiff, multiscale problems.  We have examined their order conditions, their linear stability, and techniques for solving implicit stage equations.  In \cref{app:invariants}, we have also added a short note on conservation of linear invariants.

Compared to single rate methods, the linear stability for multirate methods is much more intricate.  It not only depends on the base methods and coupling structure but also the choice of test problem.  The scalar and 2D test problems present a trade-off of generality versus simplicity to analyze.  The theoretical limitations and observed degradation of multirate stability often come from problems that are oscillatory.  These problems are challenging because the error introduced by the coupling is not damped by any partition.  In addition, we found that forgoing internal consistency can improve stability but increases the number of order conditions and limits the stage order to zero.

The coupling structure of MrGARK methods has a significant effect on the computational cost of the Newton iterations.  Decoupled methods are the cheapest and simplest to implement, especially for component partitioned problems.  Coupled methods have the potential to become prohibitively expensive but can be implemented efficiently by exploiting stage reducibility or low rank structure in the method.

The GARK framework provides new insight into the \cmrgark{} methods.  Instead of taking the approach of finding a dense output formula for coupling, we use the precise GARK order conditions.  This approach facilitated the development of methods up to order four, which to our knowledge, is the highest of this type.  Stability depends heavily on this coupling, so we derived a practical and general form for the scalar stability function.  By taking the limit as the partitions become infinitely stiff, we found a simple condition to ensure L$(\alpha)$-stability.

New standard MrGARK methods based on backward Euler and the midpoint method show excellent stability properties.  For base methods with more than one stage, however, we were unable to find methods with satisfactory stability.  Extrapolation may be the most practical way to achieve high-order, but this warrants additional investigation.
\todo{Removed irrelevant last paragraph}
\section*{Declarations}

\subsection*{Funding}

The work of S. Roberts (in part), J. Loffeld, and C.S Woodward was supported by the Lawrence Livermore Laboratory Directed Research and Development Program under tracking number 17-ERD-035.  
Their work was performed under the auspices of the U.S. Department of Energy by Lawrence Livermore National Laboratory under contract DE-AC52-07NA27344. Lawrence Livermore National Security, LLC.
LLNL-JRNL-795454.
The work of S. Roberts (in part), A. Sarshar, and A. Sandu 
was supported by the National Science Foundation (awards CCF--1613905 and ACI--1709727), Air Force Office of Scientific Research (award DDDAS FA9550-17-1-0015), and by the Computational Science Laboratory at Virginia Tech.
The work of S. Roberts (in part) was supported by the Virginia Space Grant Consortium.

This document was prepared as an account of work sponsored by an agency of the United States government. Neither the United States government nor Lawrence Livermore National Security, LLC, nor any of their employees makes any warranty, expressed or implied, or assumes any legal liability or responsibility for the accuracy, completeness, or usefulness of any information, apparatus, product, or process disclosed, or represents that its use would not infringe privately owned rights. Reference herein to any specific commercial product, process, or service by trade name, trademark, manufacturer, or otherwise does not necessarily constitute or imply its endorsement, recommendation, or favoring by the United States government or Lawrence Livermore National Security, LLC. The views and opinions of authors expressed herein do not necessarily state or reflect those of the United States government or Lawrence Livermore National Security, LLC, and shall not be used for advertising or product endorsement purposes.

\subsection*{Conflicts of Interest / Competing Interests}

The authors declare that they have no conflict of interest.

\subsection*{Availability of Data and Material}

The datasets generated during and/or analysed during the current study are available from the corresponding author on reasonable request.

\subsection*{Code Availability}

The implementation of the CUSP problem is available in ODE Test Problems: \url{https://github.com/ComputationalScienceLaboratory/ODE-Test-Problems}.

\section*{Acknowledgements}
\label{sec:acknowledgements}

The authors acknowledge Advanced Research Computing at Virginia Tech for providing computational resources and technical support that have contributed to the results reported within this paper.  We also thank Jeffrey Hittinger and Valentin Dallerit for helpful discussions.

\appendix
\section{Third Order Compound-Fast MrGARK} \label{app:mr_sdirk3}
\todo{Appendices added}

A third order \cmrgark{} method, which we will refer to as \cmrgark{} MrGARK SDIRK3, is built on the following base method of Alexander \cite{alexander1977diagonally}:
\begin{equation} \label{eq:sdirk3}
    \begin{butchertableau}{c|ccc}
         \gamma  & \gamma  & 0 & 0 \\
         \frac{\gamma }{2}+\frac{1}{2} & \frac{1}{2}-\frac{\gamma }{2} & \gamma  & 0 \\
         1 & -\frac{3 \gamma ^2}{2}+4 \gamma -\frac{1}{4} & \frac{3 \gamma ^2}{2}-5 \gamma +\frac{5}{4} & \gamma \\ \hline
         & -\frac{3 \gamma ^2}{2}+4 \gamma -\frac{1}{4} & \frac{3 \gamma ^2}{2}-5 \gamma +\frac{5}{4} & \gamma \\ \hline
         & -\frac{3 \gamma ^2}{2}+3 \gamma -\frac{1}{4} & \frac{3 \gamma ^2}{2}-3 \gamma +\frac{5}{4} & 0 \\
    \end{butchertableau}
\end{equation}
Here, $\gamma \approx 0.44$ is the middle root of $6 \gamma ^3-18 \gamma ^2+9 \gamma -1 = 0$.  The coupling coefficients are
\begin{equation} \label{eq:mr_sdirk3}
    \begin{split}
        a\FSL_{1,1} &= \frac{\left(6 \gamma ^2-24 \gamma +5\right) \left(2 \gamma ^2+2 \gamma  (\lambda -1)+(\lambda -1)^2\right)-8 \gamma ^3 M^2}{(\gamma -1) \left(6 \gamma ^2-20 \gamma +5\right) M^2} \\
        & \quad -\frac{\left(6 \gamma ^3-30 \gamma ^2-15 \gamma +5\right) M (\gamma +\lambda -1)}{(\gamma -1) \left(6 \gamma ^2-20 \gamma +5\right) M^2}, \\
        a\FSL_{1,2} &= \frac{-5 (\lambda -1)^2+12 \gamma ^4 (M-1)+4 \gamma ^3 (M-1) (3 \lambda +4 M-17)}{(\gamma -1) \left(6 \gamma ^2-20 \gamma +5\right) M^2} \\
        & \quad + \frac{\gamma ^2 \left(-6 \lambda ^2+68 \lambda +(82-72 \lambda ) M-72\right)+2 \gamma  (\lambda -1) (14 \lambda +5 M-19)}{(\gamma -1) \left(6 \gamma ^2-20 \gamma +5\right) M^2}, \\
        a\FSL_{1,3} &= -\frac{4 \gamma  \left((\lambda -1)^2+2 \gamma ^2 (M-1)^2-2 \gamma  (\lambda -1) (2 M-1)\right)}{(\gamma -1) \left(6 \gamma ^2-20 \gamma +5\right) M^2}, \\
        a\FSL_{2,1} &= -\frac{-2 \left(6 \gamma ^2-24 \gamma +5\right) (\gamma  (\lambda +1)+(\lambda -1) \lambda )+16 \gamma ^3 M^2}{2 (\gamma -1) \left(6 \gamma ^2-20 \gamma +5\right) M^2} \\
        & \quad - \frac{\left(6 \gamma ^3-30 \gamma ^2-15 \gamma +5\right) M (\gamma +2 \lambda -1)}{2 (\gamma -1) \left(6 \gamma ^2-20 \gamma +5\right) M^2}, \\
        a\FSL_{2,2} &= \frac{\gamma  \left(28 \lambda ^2-33 \lambda +5 (2 \lambda -1) M-5\right) + 2 \gamma ^3 \left(8 M^2-3 (\lambda +1)+3 (2 \lambda -7) M\right)}{(\gamma -1) \left(6 \gamma ^2-20 \gamma +5\right) M^2} \\
        & \quad + \frac{6 \gamma ^4 M -5 (\lambda -1) \lambda +\gamma ^2 \left(-6 \lambda ^2+34 \lambda +(41-72 \lambda ) M+28\right)}{(\gamma -1) \left(6 \gamma ^2-20 \gamma +5\right) M^2}, \\
        a\FSL_{2,3} &= -\frac{4 \gamma  \left((\lambda -1) \lambda +2 \gamma ^2 (M-1) M+\gamma  (\lambda +(2-4 \lambda ) M+1)\right)}{(\gamma -1) \left(6 \gamma ^2-20 \gamma +5\right) M^2}, \\
        a\FSL_{3,1} &= \frac{-36 \gamma ^5+252 \gamma ^4+20 \lambda ^2-4 \gamma ^3 \left(8 M^2+6 \lambda  M+129\right)}{4 (\gamma -1) \left(6 \gamma ^2-20 \gamma +5\right) M^2} \\
        & \quad + \frac{24 \gamma ^2 \left(\lambda ^2+5 \lambda  M+13\right)+\gamma  \left(-96 \lambda ^2+60 \lambda  M-69\right)-20 \lambda  M+5}{4 (\gamma -1) \left(6 \gamma ^2-20 \gamma +5\right) M^2}, \\
        a\FSL_{3,2} &= \frac{36 \gamma ^5-276 \gamma ^4-5 \left(4 \lambda ^2+1\right)+\gamma ^3 \left(64 M^2+48 \lambda  M+588\right)}{4 (\gamma -1) \left(6 \gamma ^2-20 \gamma +5\right) M^2} \\
        & \quad + \frac{-12 \gamma ^2 \left(2 \lambda ^2+24 \lambda  M+29\right)+\gamma  \left(112 \lambda ^2+40 \lambda  M+73\right)}{4 (\gamma -1) \left(6 \gamma ^2-20 \gamma +5\right) M^2}, \\
        a\FSL_{3,3} &= \frac{\gamma  \left(6 \gamma ^3-4 \lambda ^2-2 \gamma ^2 \left(4 M^2+9\right)+\gamma  (16 \lambda  M+9)-1\right)}{(\gamma -1) \left(6 \gamma ^2-20 \gamma +5\right) M^2}.
    \end{split}
\end{equation}

\section{Fourth Order Compound-Fast MrGARK} \label{app:mr_sdirk4}

For the \cmrgark{} method of order four, we start with a new base method, solving the coupling and base conditions together. This allows more flexibility to keep the coupling coefficients bounded functions of $\lambda$ and $M$. The following L-stable base method was derived:
\begin{equation} \label{eq:sr_sdirk4}
    \begin{butchertableau}{c|ccccc}
         \frac{1}{4} & \frac{1}{4} & 0 & 0 & 0 & 0 \\
         1 & \frac{3}{4} & \frac{1}{4} & 0 & 0 & 0 \\
         \frac{2}{5} & \frac{69}{400} & -\frac{9}{400} & \frac{1}{4} & 0 & 0 \\
         \frac{7}{11} & \frac{103241}{143748} & -\frac{1751}{71874} & -\frac{11050}{35937} & \frac{1}{4} & 0 \\
         1 & \frac{400}{459} & -\frac{35}{216} & -\frac{250}{351} & \frac{1331}{1768} & \frac{1}{4} \\ \hline
         & \frac{400}{459} & -\frac{35}{216} & -\frac{250}{351} & \frac{1331}{1768} & \frac{1}{4} \\ \hline
         & \frac{10388}{10557} & -\frac{1399}{4968} & -\frac{30425}{32292} & \frac{73205}{81328} & \frac{125}{368}
    \end{butchertableau}.
\end{equation}
When paired with the following coupling coefficients, we have the \cmrgark{} MrGARK SDIRK4 scheme:

\begin{equation} \label{eq:mr_sdirk4}
    \resizebox{\textwidth}{!}{$
    \begin{aligned}
        a\FSL_{1,1} &= \frac{-165 \left(64 \lambda ^4-192 \lambda ^3+240 \lambda ^2-148 \lambda +37\right)+2688 (4 \lambda -3) M^3-3408 \left(8 \lambda ^2-12 \lambda +5\right) M^2+448 \left(64 \lambda ^3-144 \lambda ^2+120 \lambda -37\right) M}{1836 M^4}, \\
        a\FSL_{1,2} &= \frac{1110 \left(64 \lambda ^4-192 \lambda ^3+240 \lambda ^2-148 \lambda +37\right)-1296 M^4-2292 (4 \lambda -3) M^3+8781 \left(8 \lambda ^2-12 \lambda +5\right) M^2-2101 \left(64 \lambda ^3-144 \lambda ^2+120 \lambda -37\right) M}{22464 M^4}, \\
        a\FSL_{1,3} &= -\frac{125 \left(-33 \left(64 \lambda ^4-192 \lambda ^3+240 \lambda ^2-148 \lambda +37\right)+336 (4 \lambda -3) M^3-552 \left(8 \lambda ^2-12 \lambda +5\right) M^2+83 \left(64 \lambda ^3-144 \lambda ^2+120 \lambda -37\right) M\right)}{22464 M^4}, \\
        a\FSL_{1,4} &= \frac{1331 \left(-5 \left(64 \lambda ^4-192 \lambda ^3+240 \lambda ^2-148 \lambda +37\right)+32 (4 \lambda -3) M^3-60 \left(8 \lambda ^2-12 \lambda +5\right) M^2+11 \left(64 \lambda ^3-144 \lambda ^2+120 \lambda -37\right) M\right)}{56576 M^4}, \\
        a\FSL_{1,5} &= \frac{-85 \left(64 \lambda ^4-192 \lambda ^3+240 \lambda ^2-148 \lambda +37\right)+192 M^4+16 (4 \lambda -3) M^3-648 \left(8 \lambda ^2-12 \lambda +5\right) M^2+175 \left(64 \lambda ^3-144 \lambda ^2+120 \lambda -37\right) M}{3328 M^4}, \\
        a\FSL_{2,1} &= \frac{-165 \left(64 \lambda ^4-48 \lambda ^2+68 \lambda -17\right)+10752 \lambda  M^3-3408 \left(8 \lambda ^2-1\right) M^2+448 \left(64 \lambda ^3-24 \lambda +17\right) M}{1836 M^4}, \\
        a\FSL_{2,2} &= \frac{1110 \left(64 \lambda ^4-48 \lambda ^2+68 \lambda -17\right)-1296 M^4-9168 \lambda  M^3+8781 \left(8 \lambda ^2-1\right) M^2-2101 \left(64 \lambda ^3-24 \lambda +17\right) M}{22464 M^4}, \\
        a\FSL_{2,3} &= -\frac{125 \left(-33 \left(64 \lambda ^4-48 \lambda ^2+68 \lambda -17\right)+1344 \lambda  M^3-552 \left(8 \lambda ^2-1\right) M^2+83 \left(64 \lambda ^3-24 \lambda +17\right) M\right)}{22464 M^4}, \\
        a\FSL_{2,4} &= \frac{1331 \left(5 \left(-64 \lambda ^4+48 \lambda ^2-68 \lambda +17\right)+128 \lambda  M^3-60 \left(8 \lambda ^2-1\right) M^2+11 \left(64 \lambda ^3-24 \lambda +17\right) M\right)}{56576 M^4}, \\
        a\FSL_{2,5} &= \frac{-85 \left(64 \lambda ^4-48 \lambda ^2+68 \lambda -17\right)+192 M^4+64 \lambda  M^3-648 \left(8 \lambda ^2-1\right) M^2+175 \left(64 \lambda ^3-24 \lambda +17\right) M}{3328 M^4}, \\
        a\FSL_{3,1} &= \frac{-33 \left(32000 \lambda ^4-76800 \lambda ^3+84720 \lambda ^2-56180 \lambda +15773\right)+215040 (5 \lambda -3) M^3}{183600 M^4} \\
        & \quad + \frac{-3408 \left(800 \lambda ^2-960 \lambda +353\right) M^2+448 \left(6400 \lambda ^3-11520 \lambda ^2+8472 \lambda -2809\right) M}{183600 M^4}, \\
        a\FSL_{3,2} &= \frac{222 \left(32000 \lambda ^4-76800 \lambda ^3+84720 \lambda ^2-56180 \lambda +15773\right)-129600 M^4-183360 (5 \lambda -3) M^3}{2246400 M^4} \\
        & \quad + \frac{8781 \left(800 \lambda ^2-960 \lambda +353\right) M^2-2101 \left(6400 \lambda ^3-11520 \lambda ^2+8472 \lambda -2809\right) M}{2246400 M^4}, \\
        a\FSL_{3,3} &= \frac{33 \left(32000 \lambda ^4-76800 \lambda ^3+84720 \lambda ^2-56180 \lambda +15773\right)-134400 (5 \lambda -3) M^3}{89856 M^4} \\
        & \quad + \frac{2760 \left(800 \lambda ^2-960 \lambda +353\right) M^2-415 \left(6400 \lambda ^3-11520 \lambda ^2+8472 \lambda -2809\right) M}{89856 M^4}, \\
        a\FSL_{3,4} &= \frac{1331 \left(-32000 \lambda ^4+76800 \lambda ^3-84720 \lambda ^2+56180 \lambda +2560 (5 \lambda -3) M^3\right)}{5657600 M^4} \\
        & \quad + \frac{1331 \left(-60 \left(800 \lambda ^2-960 \lambda +353\right) M^2+11 \left(6400 \lambda ^3-11520 \lambda ^2+8472 \lambda -2809\right) M-15773\right)}{5657600 M^4}, \\
        a\FSL_{3,5} &= \frac{-17 \left(32000 \lambda ^4-76800 \lambda ^3+84720 \lambda ^2-56180 \lambda +15773\right)+19200 M^4+1280 (5 \lambda -3) M^3}{332800 M^4} \\
        & \quad + \frac{-648 \left(800 \lambda ^2-960 \lambda +353\right) M^2+175 \left(6400 \lambda ^3-11520 \lambda ^2+8472 \lambda -2809\right) M}{332800 M^4}, \\
        a\FSL_{4,1} &= \frac{-33 \left(425920 \lambda ^4-619520 \lambda ^3+280720 \lambda ^2-35660 \lambda +2387\right)+1300992 (11 \lambda -4) M^3}{2443716 M^4} \\
        & \quad + \frac{-137456 \left(264 \lambda ^2-192 \lambda +29\right) M^2+448 \left(85184 \lambda ^3-92928 \lambda ^2+28072 \lambda -1783\right) M}{2443716 M^4}, \\
        a\FSL_{4,2} &= \frac{222 \left(425920 \lambda ^4-619520 \lambda ^3+280720 \lambda ^2-35660 \lambda +2387\right)-1724976 M^4-1109328 (11 \lambda -4) M^3}{29899584 M^4} \\
        & \quad + \frac{354167 \left(264 \lambda ^2-192 \lambda +29\right) M^2-2101 \left(85184 \lambda ^3-92928 \lambda ^2+28072 \lambda -1783\right) M}{29899584 M^4}, \\
        a\FSL_{4,3} &= -\frac{25 \left(-33 \left(425920 \lambda ^4-619520 \lambda ^3+280720 \lambda ^2-35660 \lambda +2387\right)+813120 (11 \lambda -4) M^3\right)}{29899584 M^4} \\
        & \quad - \frac{25 \left(-111320 \left(264 \lambda ^2-192 \lambda +29\right) M^2+415 \left(85184 \lambda ^3-92928 \lambda ^2+28072 \lambda -1783\right) M\right)}{29899584 M^4}, \\
        a\FSL_{4,4} &= \frac{-425920 \lambda ^4+619520 \lambda ^3-280720 \lambda ^2+35660 \lambda +15488 (11 \lambda -4) M^3}{56576 M^4} \\
        & \quad + \frac{-2420 \left(264 \lambda ^2-192 \lambda +29\right) M^2+11 \left(85184 \lambda ^3-92928 \lambda ^2+28072 \lambda -1783\right) M-2387}{56576 M^4}, \\
        a\FSL_{4,5} &= \frac{-17 \left(425920 \lambda ^4-619520 \lambda ^3+280720 \lambda ^2-35660 \lambda +2387\right)+255552 M^4+7744 (11 \lambda -4) M^3}{4429568 M^4} \\
        & \quad + \frac{-26136 \left(264 \lambda ^2-192 \lambda +29\right) M^2+175 \left(85184 \lambda ^3-92928 \lambda ^2+28072 \lambda -1783\right) M}{4429568 M^4}, \\
        a\FSL_{5,1} &= \frac{16 \lambda  \left(-165 \lambda ^3+168 M^3-426 \lambda  M^2+448 \lambda ^2 M\right)}{459 M^4}, \\
        a\FSL_{5,2} &= -\frac{-8880 \lambda ^4+162 M^4+1146 \lambda  M^3-8781 \lambda ^2 M^2+16808 \lambda ^3 M}{2808 M^4}, \\
        a\FSL_{5,3} &= \frac{125 \lambda  \left(33 \lambda ^3-21 M^3+69 \lambda  M^2-83 \lambda ^2 M\right)}{351 M^4}, \\
        a\FSL_{5,4} &= \frac{1331 \lambda  \left(-10 \lambda ^3+4 M^3-15 \lambda  M^2+22 \lambda ^2 M\right)}{1768 M^4}, \\
        a\FSL_{5,5} &= \frac{-85 \lambda ^4+3 M^4+\lambda  M^3-81 \lambda ^2 M^2+175 \lambda ^3 M}{52 M^4}.
    \end{aligned}
    $}
\end{equation}

\section{Conservation of Linear Invariants} \label{app:invariants}

For some applications, it is desirable that a numerical integrator uphold invariant properties of the system, such as conservation of mass and energy. These invariants are generally expressed as \emph{first integrals} of the system. It is well known that non-partitioned explicit and implicit Runge--Kutta methods conserve linear first integrals but must meet certain restrictions to conserve quadratic ones, e.g. as with symplectic methods. Partitioned Runge--Kutta methods must obey additional restrictions to conserve even linear invariants. A detailed discussion about preservation of first integrals by Runge--Kutta methods can be found in \cite{Hairer2006}.

GARK schemes are partitioned methods, and here we briefly describe conditions for them to uphold linear first integrals.  Consider an ODE \cref{eq:ode} satisfying the following linear invariant:
\begin{equation} \label{eq:linear_invariant}
    \zeta^T \left( f\F(y) + f\S(y) \right) = 0
    \quad \Rightarrow \quad
    \zeta^T y(t) = \text{const}
\end{equation}
When a GARK method applied to this system, the step update \cref{eq:GARK} satisfies
\begin{equation} \label{eq:gark_linear_invariant}
    \zeta^T y_{n+1} = \zeta^T y_{n} + H \, \sum_{j=1}^{s\F} b\F_j \,  \zeta^T  \, f\F \mleft( Y\F_j \mright) + H \, \sum_{j=1}^{s\S} b\S_j \,  \zeta^T \, f\S \mleft( Y\S_j \mright)
\end{equation}
In order to apply \cref{eq:linear_invariant}, the arguments of $f\F$ and $f\S$ must be identical.  In general, the arguments in \cref{eq:gark_linear_invariant} are different: $Y\F_j \neq Y\S_j$.  Moreover, the number of slow stages can be different than the number of fast stages, which prevents the pairing of terms as in \cref{eq:linear_invariant}.  Therefore, a general GARK method cannot be expected to preserve linear invariants.

There are special cases, however, where is is possible.  If the subsystems individually satisfy
\begin{equation*}
    \zeta^T f\F(y) = 0,
    \quad \text{and} \quad
    \zeta^T f\S(y) = 0,
\end{equation*}
one can see \cref{eq:gark_linear_invariant} simplifies to $\zeta^T y_{n+1} = \zeta^T y_{n}$.  Also, when
\begin{equation} \label{eq:GARK_ARK_special_case}
    \mathbf{A}\FF = \mathbf{A}\SF,
    \quad
    \mathbf{A}\FS = \mathbf{A}\SS,
    \quad
    \mathbf{b}\F = \mathbf{b}\S,
\end{equation}
we can achieve $s\F$ = $s\S$ and $Y\F_j = Y\S_j$.  In fact, this GARK method degenerates into a special class of partitioned Runge--Kutta schemes known to preserve linear invariants \cite{Hairer2006}.  The multirate methods in \cite{Sandu_2007_MR_RK2}, for example, satisfy \cref{eq:GARK_ARK_special_case}.

\bibliographystyle{spmpsci}
\bibliography{main}

\end{document}